\documentclass[12pt]{amsart}
\usepackage[english]{babel}
\usepackage[utf8]{inputenc}
\usepackage[T1]{fontenc}
\usepackage{hyperref}
\usepackage{dsfont}
\usepackage{color}

\usepackage{stmaryrd}
\usepackage{cite}

\usepackage{pgf,tikz,pgfplots}
\pgfplotsset{compat=1.14}
\usepackage{mathrsfs}
\usetikzlibrary{arrows}

\usepackage{geometry}
\allowdisplaybreaks

\usepackage{amsmath}
\usepackage{amssymb}
\usepackage{amsthm}
\usepackage{amsfonts}

\newcommand{\Id}{\mathrm{Id}}

\newcommand{\R}{\mathbb{R}}

\renewcommand{\div}{\mathrm{div}}

\newcommand{\G}{\mathcal{G}}
\newcommand{\F}{\mathcal{F}}
\newcommand{\A}{\mathcal{A}}

\newcommand{\N}{\mathbf{N}}

\newcommand{\Ur}{U^r}

\newcommand{\tw}{\Tilde{w}}
\newcommand{\hw}{\hat{w}}

\newcommand{\Ve}{\mathbf{V}^e}
\newcommand{\Vi}{\mathbf{V}^i}
\newcommand{\tiV}{\mathbf{\Tilde{V}}}

\newcommand{\nabs}{\nabla^{\sigma}}

\newcommand{\Xp}{\Dot{X}}

\newcommand{\weing}{\mathrm{L}}

\newcommand{\tub}{\mathcal{T}_0}
\newcommand{\tubp}{\mathcal{T}_1}
\newcommand{\norz}{\mathbf{n}_0}
\newcommand{\Nor}{\mathbf{N}}
\newcommand{\nor}{\mathbf{n}}

\newcommand{\rhop}{\Dot{\rho}}
\newcommand{\sigmap}{\Dot{\sigma}}

\theoremstyle{plain}
\newtheorem{thm}{Theorem}[section]
\newtheorem{lem}[thm]{Lemma}
\newtheorem{prop}[thm]{Proposition}
\newtheorem{cor}[thm]{Corollary}

\theoremstyle{definition}
\newtheorem{defn}{Definition}[section]

\theoremstyle{remark}
\newtheorem{rem}[thm]{Remark}

\title[Shape derivatives for the DtN and cellular protrusion]{Shape derivative for Dirichlet-to-Neumann Operator on a manifold and application to cellular protrusion}
 \author{F.Noisette}
  \date{\today}

\begin{document}

\begin{abstract}
We establish a shape-derivative formula for the Dirichlet-to-Neumann operator on a compact manifold.
Then, we apply this formula to obtain the well-posedness in $H^1$ under a specific Rayleigh-Taylor condition to an equation describing cell protrusions.
This equation is a generalisation of the theoretical part of \cite{Cellule1} to any -2D as well as 3D- surfaces.
\end{abstract}

\maketitle

\tableofcontents

{
\small	
\textbf{\textit{Keywords:}} Dirichlet-to-Neumann operator, Free boundary problems, Shape derivative, Microbiology, Cell protrusions 
}

\section{Introduction}

\subsection{The biological context}

\textit{Invadopodium} is a name given to a phenomenon through which cancerous cells enter the bloodstream.
The cells in question develop a long protrusion, which can fit inside small holes, and allows the cell to gain a foothold there.
The creation of those protrusions is the result of chemical reactions both inside and outside the cell.
The boundary of the cell produces some enzymes that degrades the extra-cellular matrix (ECM) in which the cell lies.
In response to this degradation, a signal propagates in the cell, causing \textit{actin} to polymerize.
When actin polymerizes, it forms long filaments that push the boundary of the cell and create the protrusion.
The model we present here is a simplified model, but with the complexity of the phenomena at stake, it is useful to have toy models, that isolate the theoretical and numerical difficulties.
We refer to \cite{Holmes-Edelstein} for a general survey on cell motility.
A reader willing to investigate the biology of the phenomena in question can read \cite{Bio1} for  experimental examples of cell invadopodia and \cite{Bio2} for experimental examples of cell migrations.

On a mathematical standpoint, following \cite{Cellule1} and \cite{Cellule2}, we adopt a very simple model: the degradation $c^*$ of the ECM as well as the signal $\sigma$ inside the cell are at diffusion equilibrium
\begin{equation*}
\Delta c^* = 0, \quad \text{inside the ECM},
\end{equation*}
and 
\begin{equation*}
\Delta \sigma = 0, \quad \text{inside the cell}.
\end{equation*}
Those elliptics problems are coupled with boundary conditions.
We take as datum of the problem the flux $\Psi$ of (MT1-MMPs), which is the enzyme responsible for the degradation of the ECM:
\begin{equation*}
\partial_n c^* = \Psi, \quad \text{ on the boundary of the cell}.
\end{equation*}
The coupling condition between the interior and the exterior is given by the fact that the signal inside the cell is generated by the perceived degradation of the ECM:
\begin{equation*}
\sigma = c^*, \quad \text{ on the boundary of the cell}.
\end{equation*}
And finally the cell moves according to the signal
\begin{equation*}
v = \nabla\sigma, \quad \text{ on the boundary of the cell}.
\end{equation*}

\begin{figure} \centering  
\includegraphics[scale= 0.5]{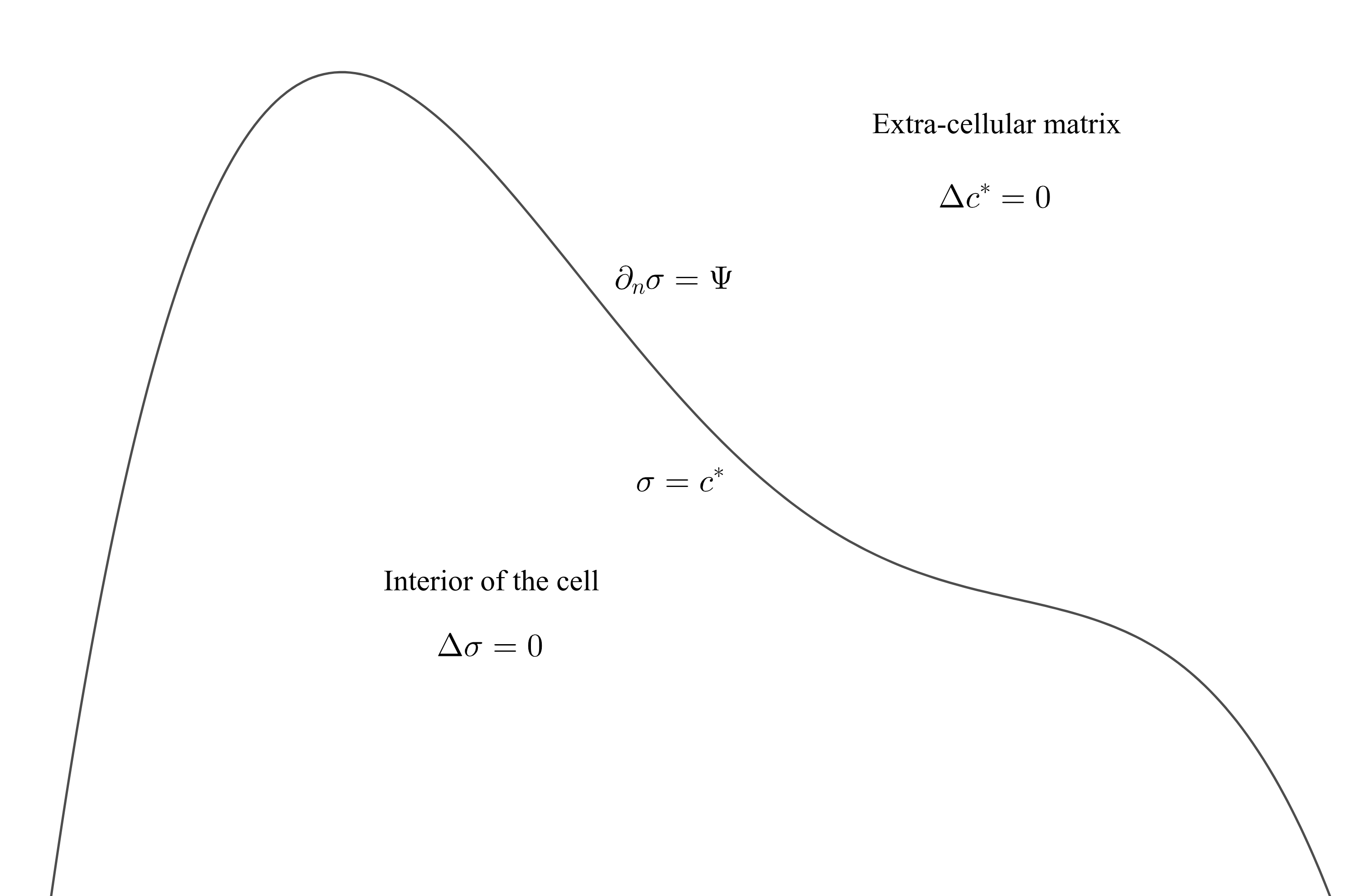}
\caption{Schematic description of the biological model.} \label{fig3} 
\end{figure}

The next logical step in term of modelisation is to introduce the Dirichlet-to-Neumann operator.
This operator associates to a function $\Psi$ on the boundary of a domain the normal derivative of its harmonic extension on this domain:
\begin{subequations}
\begin{align}
\Delta \Phi &= 0,  \\
\Phi_{|\partial\Omega} &= \Psi,
\end{align}
\end{subequations}
and the Dirichlet-to-Neumann operator is given by
\begin{equation}
\G[\Omega]\Psi := \partial_n\Phi, \quad \text{on } \partial \Omega.
\end{equation}
In this paper, the surface of the cell $\Gamma$ will be parameterised in normal-tangential coodirnates as a graph over a reference surface $\Gamma_0$ (see Fig \ref{fig1}).
A given surface correspond to two domains, the interior of the cell $\Omega^i$ and the exterior $\Omega^e$.
Instead of $\G[\Omega^e]$ and $\G[\Omega^i]$, by abuse of notation, we will write $\G^e[\rho]$ and $\G^i[\rho]$, 
\begin{equation*}
\G^e[\rho] := \sqrt{1+|\nabla \rho|^2} \G[\Omega^e],
\quad \text{and} \quad
\G^i[\rho] := \sqrt{1+|\nabla \rho|^2} \G[\Omega^i],
\end{equation*}
where $\rho$ is the parametrisation of $\Gamma$ and $\sqrt{1+|\nabla \rho|^2}$ is a renormalisation due to the fact that one works on the reference domain instead of the real one.
This operator intervenes naturally in our equations since both the signal $\sigma$ and the deterioration $c^*$ are harmonic.
Therefore, the study of our model can be reduced to the study of the Dirichlet-to-Neumann operator as well as its inverse.
The evolution equation  for the parametrization $\rho$ of the boundary of the cell reads -more or less- as 
\begin{equation}\label{e:prot-equa-intro}
\partial_t\rho - \A[\rho](|N|\psi) = 0,
\end{equation}
where $|\Nor| = \sqrt{1+|\nabla \rho|^2}$ is the norm of the non-unit vector used to describe the Dirichlet to Neumann operator and   the operator $\A$ is the composition of the Dirichlet-to-Neumann operator of the inside of the cell with the Neumann-to-Dirichlet operator (its inverse) on the outside of the cell,
\begin{equation}
\A[\rho](\psi) = \G^i[\rho]\left(\left(\G^e[\rho]\right)^{-1}(\psi)\right), 
\end{equation}
and the function $\psi$ is a datum of the problem, namely, the flux on the boundary of enzymes produced by the cell that cause the degradation of the extra cellular matrix.
The difficulty here lies in that the unknown of the equation is the boundary of the cell -or in our case its parametrization $\rho$- and therefore \eqref{e:prot-equa-intro} is completely non-linear and non-local.

Several free boundary problems in fluid mechanic have a structure similar to \eqref{e:prot-equa-intro}.
The most famous one is maybe the water-wave equation in Zhakarov-Craig-Sulem formulation (introduced in \cite{WW-init2} and \cite{WW-init1}), that describes irrotationnal incompressible water waves under the influence of gravity only:
\begin{subequations}
\begin{align}
\partial_t \zeta 
    - \G[\zeta]\psi 
    &= 0, \\
\partial_t\psi 
    + g\zeta
    +\tfrac{1}{2}|\nabla \psi|^2
    -\tfrac{1}{2} \frac{\left(\G[\zeta]\psi + \nabla\zeta\cdot\nabla\psi\right)^2}{1+|\nabla\zeta|^2}
    &= 0,
\end{align}
\end{subequations}
where $\zeta$ is the surface elevation and $\psi$ the potential of the velocity on the surface.
Another equation famously linked to a free boundary problem is the Hele-Shaw equation.
It corresponds to the scenario where the fluid velocity and pressure are related through Darcy's law, namely
\begin{equation}
u = -\nabla (P + gz),
\end{equation}
which combined with the free boundary condition on the free surface gives 
\begin{equation}
\partial_t\zeta + g\, \G[\zeta](\zeta) = 0.
\end{equation}
Finally, we mention the Muskat equation (introduced for the first time in \cite{MuskatInit}), that describe the interface of two fluids with different densities, both obeying Darcy's law.
It writes as follows
\begin{equation}
\partial_t f + \mathcal{L}(f) f = 0,
\end{equation}
where the operator $\mathcal{L}$ is defined -in 3D- through 
\begin{equation}
\mathcal{L}(f)g 
    = -\tfrac{1}{2\pi} \int_{\R^2}{\tfrac{\alpha\cdot \big(\nabla g(x+\alpha)-\nabla g(x)\big)}{\left(|\alpha|^2 + |f(x+\alpha)-f(x)|^2 \right)^{3/2}} d\alpha}.
\end{equation}
Both equations convey the same kind of properties as \eqref{e:prot-equa-intro}: the main operator is non-local, non-linear and not of order $0$.
A good strategy to start in this case is to look at the linearised equation around a fixed state.
We refer to \cite{Lannes} for the case of WW and \cite{QuasiLinMusk3D} for the case of $3D$-Muskat.
Alazard wrote recently a good survey \cite{ParaLinSurvey} on this kind of problematics.

When computing the linearisation of \eqref{e:prot-equa-intro}, the crucial point is to understand the limit 
\begin{equation*}
\lim_{\varepsilon\rightarrow 0}{\tfrac{\A[\rho + \varepsilon\rhop] -\A[\rho]}{\varepsilon}},
\end{equation*}
that is to compute the derivative of $\A$ with respect to the parametrisation of the domain, or \textit{shape derivative} of the operator $\A$.
But as we said, the operator $\A$ is the composition of Dirichlet-to-Neumann on the interior domain with the Neumann-to-Dirichlet operator on the exterior domain.
Thus, this shape derivative can be reduced to the shape derivative of the Dirichlet-to-Neumann operator itself.
This is a problem that was already explored, in the case of the \textit{strip} $\mathcal{S}:=\R^d\times[0,1]$ for water waves (see \cite{Lannes} for the shape derivative and \cite{ParaLinDtN} for the paralinearization, which is the next step at low regularity).
On a generic manifold, the shape derivative of the Dirichlet to Neumann operator writes as follows:
\begin{equation} \label{e:shape-der-form-intro}
\mathrm{d}_{\rho}\G(\rhop)\psi 
    = -\G(\rhop w) 
    - \tfrac{1}{\det(K)}\div_{\Gamma_0}(\rhop \mathbf{V})
    + \rhop\, a_0,
\end{equation}
where $w$, $\mathbf{V}$, $a_0$ and $K$ are functions of $\psi$ and $\rho$, but independent of $\rhop$.
This formula, which is a useful result out of itself,  is similar to the one on $\mathcal{S}$, except for the last term $\rhop\, a_0$, which is due to the curvature of the boundary of the reference domain $\Gamma_0$: the formulas for $w$ and $\mathbf{V}$ are also slightly different for the same reason.

In \cite{Cellule1} and \cite{Cellule2}, Gallinato, Poignard and their co-authors studied  \eqref{e:prot-equa-intro}.
Their aim was mainly to produce numerical simulations for the biological model presented above. 
They also provided a justification for the well-posedness of the linearization of \eqref{e:prot-equa-intro} around the circle.
Our generalization has two new difficulties. 
First of, since we are not working in $2D$, we don't have the tools of complex analysis for integro-differential operators (like Dirichlet-to-Neumann).
Secondly, since we are on a generic manifold, we don't have access to Fourier series.

\subsection{Description of the Chapter and sketch of the proofs}

This Chapter has three main objectives.

We start, in Section \ref{s:shape-der} by computing the shape derivative of the Dirichlet to Neumann operator on a compact manifold.
The Paragraphs \ref{ss:diff-geom}, \ref{ss:adm-diff} and \ref{ss:def-DtN} are devoted to the definition of the Dirichlet to Neumann operator on a moving domain $\Omega$, but seen from a fixed boundary $\Gamma_0$:
in \ref{ss:diff-geom}, we recall the definition of tubular neighborhoods and tangential and normal coordinates and derivatives;
in \ref{ss:adm-diff}, we construct a diffeomorphism in order to work on an artificially fixed domain $\Omega_0$, even if the domain $\Omega$ is not;
in \ref{ss:def-DtN} we construct the Dirichlet-to-Neumann operator.
Then, in Paragraph \ref{ss:shape-der-DtN}, we prove that the Dirichlet-to-Neumann operator is differentiable on the open set $\Ur$, introduced in Paragraph \ref{ss:adm-diff}, and give an explicit formula for its shape derivative in Theorem \ref{t:shape-der-formula}.
The ideas behind the proof of this formula are to use \textit{Alinhac's good unknown}, as well as the fact that, due to harmonicity, even on the boundary, one has $\Delta\Phi = 0$, which allows to cancel normal derivatives with tangential ones.

Then, in Section \ref{s:Lin-Equa}, we linearize \eqref{e:prot-equa-intro} around $0$.
We start by recalling precisely the equation at stake in Paragraphs  \ref{ss:remind-bio} and \ref{ss:evol-equa}.
Then, in Paragraph \ref{ss:lin-equa}, we deduce from Theorem \ref{t:shape-der-formula}, the shape derivative for the operator $\A$ and linearize the equation \eqref{e:prot-equa-intro}.

Finally, in Section \ref{s:wp}, we prove the well-posedness in $H^1(\Gamma_0)$ of the linearized equation \eqref{e:prot-linearised}.
In Paragraph \ref{ss:reg-prop}, we prove some regularity property of $(\Id-\A_0)$, that we use to simplify the skew-symmetric terms of \eqref{e:prot-linearised}.
In Paragraph \ref{ss:L2-est}, we prove the $L^2$ estimate using the coercivity of Dirichlet-to-Neumann.
In Paragraph \ref{ss:H1-est}, we do the same for the $H^1$ estimate, except that we have to avoid some difficulties with regard to commutations.
Finally, in Paragraph \ref{ss:wp}, we prove well-posedness via a Leray scheme using the result of \ref{ss:L2-est} and \ref{ss:H1-est}.


\section{Shape derivative of the Dirichlet-to-Neumann operator} \label{s:shape-der}


\subsection{Reminders of differential geometry} \label{ss:diff-geom}

Let $d\in\mathbb{N}$ and  $\Omega_0 \subset \R^{d+1}$ be a smooth bounded domain.
We denote by $\Gamma_0$ its boundary 
$$\Gamma_0 := \partial\Omega_0.$$
As $\Omega_0$ is smooth, its boundary admits an open \textit{tubular neighborhood} $\tub$, 
that is, there exists a neighborhood of $\Gamma_0$ such that each point $Y\in \tub$ can be uniquely represented as $Y=X+x\norz(X)$, with $X\in \Gamma_0$, $x\in \R$ and $\norz(X)$ the exterior normal vector to $\Gamma_0$ at point $X$.

Let $\Omega$ be a Lipschitz domain, with $\partial\Omega \subset \tub$.
We denote $\Gamma:=\partial\Omega$.
We also assume that $\Gamma$ admits a \textit{normal parametrization} over $\Gamma_0$, that is, there exists a Lipschitz function $\rho$ such that:
\begin{equation*}
\Gamma = \{X+\rho(X)\norz(X) \,;\, X\in \Gamma_0\}.
\end{equation*}
In particular, if we denote by $\varepsilon_0$ the \textit{size} of $\tub$, that is
\begin{equation}
\varepsilon_0 := \min_{X\in \Gamma_0}\sup\{x\in \R_+; X-x\norz(X)\in \tub \text{ and } X+x\norz(X)\in \tub\},
\end{equation}
we have 
\begin{equation}
\forall X\in \Gamma_0, \quad |\rho(X)| \leq \varepsilon_0.
\end{equation}

\begin{figure} \centering  
\includegraphics[scale= 0.8]{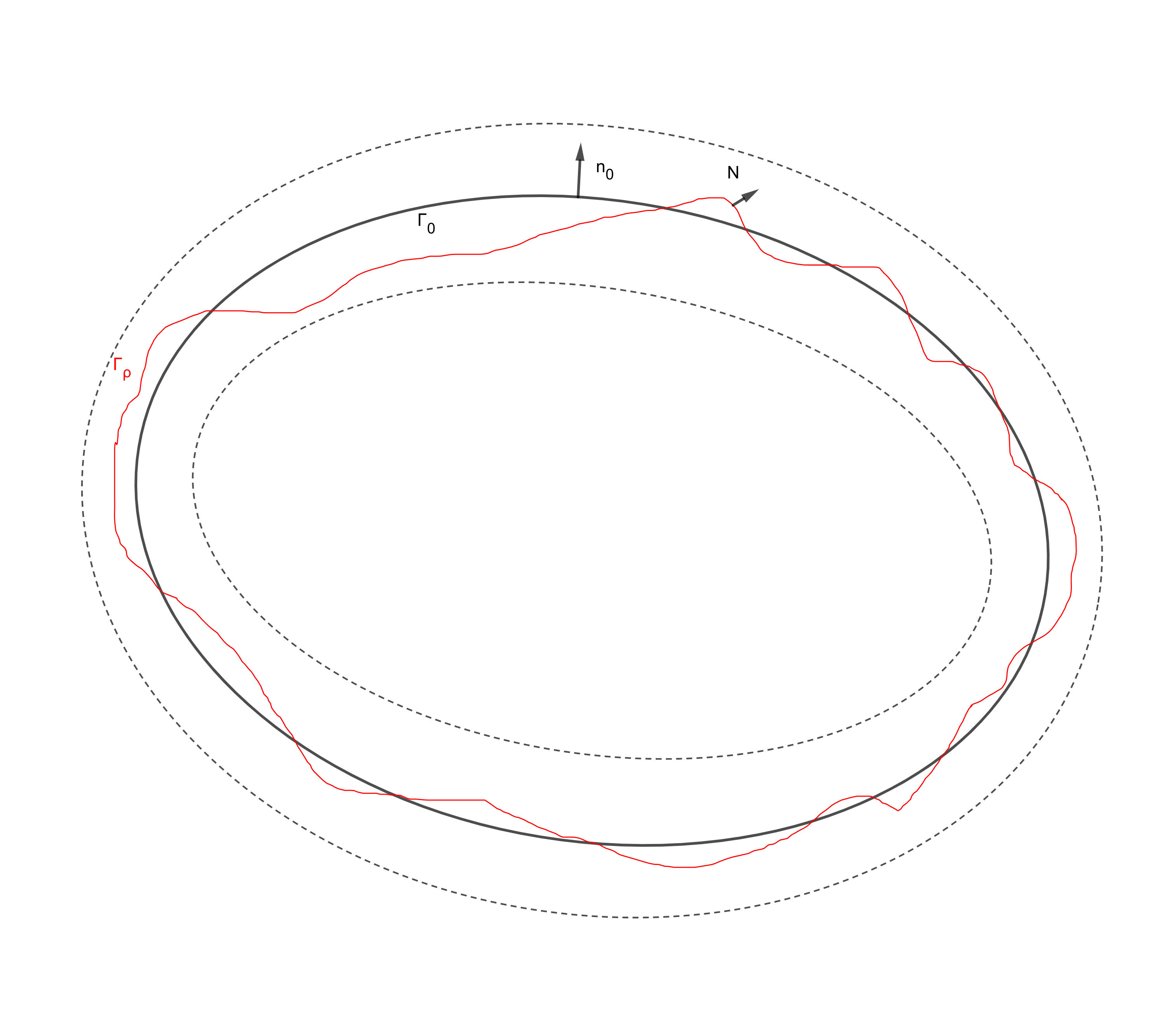}
 \caption{Parametrization of the cell over a reference cell $\Omega_0$:  In black the boundary $\Gamma_0$ of the reference cell, the dotted line represent the boundaries of the tubular neighborhood $\tub$ and the red line represent $\Gamma(\rho)$, the boundary of $\Omega$} \label{fig1} 
\end{figure}

Let us introduce the tangential gradient and divergence operators on the tubular neighborhood $\tub$.

\begin{defn}
Let $f\in C^1(\tub,\R)$ and $\mathbf{v}\in C^1(\tub,\R^{d+1})$, we denote:
\begin{align}
\nabla_{\Gamma_0} f &:= \nabla f - (\norz\cdot \nabla)f\, \norz, \\
\div_{\Gamma_0}(\mathbf{v}) &:= \div(\mathbf{v}) - \mathrm{d}\mathbf{v}(\norz)\cdot \norz.
\end{align}
\end{defn}

\begin{prop}
Let $f,g\in C^1(\tub,\R)$ and $\mathbf{v},\mathbf{w}\in C^1(\tub,\R^{d+1})$, we have the following:
\begin{enumerate}
\item[(i)]
if $f_{|\Gamma_0} = g_{|\Gamma_0}$, then 
\begin{equation}
(\nabla_{\Gamma_0}f)_{|\Gamma_0}
    =(\nabla_{\Gamma_0}g)_{|\Gamma_0},
\end{equation}
\item[(ii)]
if $\mathbf{v}_{|\Gamma_0} = \mathbf{w}_{|\Gamma_0}$, then 
\begin{equation}
(\div_{\Gamma_0}\mathbf{v})_{|\Gamma_0}
    =(\div_{\Gamma_0}\mathbf{w})_{|\Gamma_0}.
\end{equation}
\end{enumerate}
\end{prop}

\begin{proof}
Let us fix $X\in \Gamma_0$ and $u\in T_X\Gamma_0$.
Since $f-g$ and $\mathbf{v}-\mathbf{w}$ are constant equal to zero on $\Gamma_0$, one has:
\begin{equation}
\nabla (f-g)(X) \cdot \mathbf{u} = 0,
\quad \text{and,} \quad
\mathrm{d}_X(\mathbf{v}-\mathbf{w})(\mathbf{u}) \cdot \mathbf{u} =0.
\end{equation}
Therefore 
\begin{equation}
\nabla_{\Gamma_0} (f-g)(X) = 0,
\end{equation}
and 
\begin{equation}
\mathrm{d}_X(\mathbf{v}-\mathbf{w})(\norz) \cdot \norz 
    = \div(\mathbf{v}-\mathbf{w})(X),
\end{equation}
which leads to 
\begin{equation}
\div_{\Gamma_0}(\mathbf{v})(X) = \div_{\Gamma_0}(\mathbf{w})(X).
\end{equation}
\end{proof}

A useful consequence of this proposition is that it allows to define the tangential gradient and divergence for functions and vector fields defined on the boundary $\Gamma_0$.

\begin{defn}
Let $f\in C^1(\Gamma_0,\R)$, we define its tangential gradient $\nabla_{\Gamma_0}f$ as $\left(\nabla_{\Gamma_0}f^{\dagger}\right)_{|\Gamma_0}$, where $f^{\dagger}$ is any extension of $f$ to $\tub$.
Similarly, for $\mathbf{v}\in C^1(\Gamma_0, \R^{d+1})$, its divergence on $\Gamma_0$ is defined as $\div_{\Gamma_0}(\mathbf{v}^{\dagger})_{|\Gamma_0}$, where $\mathbf{v}^{\dagger}$ is any extension of $\mathbf{v}$ to $\tub$.
\end{defn}

One has basic calculus rules on those operators.
In particular, using the fact that 
\begin{equation}
d_X(f\mathbf{v})(\norz(X)) 
    = (\norz(X)\cdot \nabla f(X)) v(X)
     + f(X) d_X\mathbf{v}(\norz(X)),
\end{equation}
one gets
\begin{equation} \label{e:divg-prod}
\div_{\Gamma_0}(f\mathbf{v}) = f\,\div_{\Gamma_0}(\mathbf{v}) + \nabla_{\Gamma_0}f \cdot \mathbf{v},
\end{equation}
and on the boundary, we have
\begin{equation} \label{e:mean-curv}
\div_{\Gamma_0}(f\norz) = H_0 f,
\end{equation}
where $H_0:= \div_{\Gamma_0}(\norz)$ is the \textit{mean curvature} of $\Gamma_0$.

For $X\in\Gamma_0$, we denote by $\Pi_0(X)$ the orthogonal projector onto $T_{X}\Gamma_0$ and $\Pi_0^{\perp}(X)$ the orthogonal projector onto $\R\norz(X)$,
\begin{align}\label{art4-e:Pi0_def}
\Pi_0(X)& :\mathbf{u}\in \R^{d+1}\mapsto \mathbf{u} - (\mathbf{u}\cdot\norz(X))\norz(X) \in  T_{X}\Gamma_0, \\
\Pi_0^{\perp}(X)& :\mathbf{u}\in \R^{d+1}\mapsto (\mathbf{u}\cdot\norz(X))\norz(X) \in  \R\norz(X).
\end{align}
We admit that $\div_{\Gamma_0}$ corresponds to the divergence operator on $\Gamma_0$.
In particular, we have:
\begin{thm}[Stokes]
For $\mathbf{v}\in C^1(\Gamma_0,\R^{d+1})$, one has the formula:
\begin{equation}
\int_{\Gamma_0}{\div_{\Gamma_0}(\Pi_0 \mathbf{v})} =0.
\end{equation}
\end{thm}

Let us also denote by $\partial_{\norz}$ the normal derivative to $\Gamma_0$
\begin{equation}
\partial_{\norz}f := (\norz\cdot\nabla)f.
\end{equation}
We denote by $\weing$ the Weingarten map:
\begin{equation}\label{e:def-weingarten}
\forall X\in \Gamma_0, \forall H\in T_X\Gamma_0, \qquad 
\weing(X)(H) := d_X\norz(H).
\end{equation}

\subsection{Construction of an admissible diffeomorphism} \label{ss:adm-diff}

The goal of this paragraph is to construct (under suitable hypothesis) an \textit{admissible} diffeomorphism $\sigma$ sending $\Omega_0$ to $\Omega$. 
The conditions to be admissible, listed in Definition \ref{d:adm-diffeo}, are here to ensure that the elliptic problem with the laplacian on $\Omega$ can be transported on an elliptic problem on $\Omega_0$ via $\sigma$.
There is also an additional condition on the neighborhood of the boundary to ensure that all the computation later on run smoothly.

\begin{defn} \label{d:adm-diffeo}
Let $r\geq 1$ be an integer.
Let $\rho\in W^{r,\infty}(\Gamma_0)$ and $\Omega$ be the domain associated with $\rho$. 
We say that a diffeomorphism $\sigma : \Omega_0 \rightarrow \Omega$ is admissible when the following conditions are satisfied
\begin{enumerate}
\item 
there exists a neighborhood $\tubp\subset \tub$ of $\Gamma_0$ such that 
\begin{equation} \label{e:diff-bound-val}
\forall X+x\norz(X)\in \tubp, \quad
 \sigma(X+x\norz(X)) = X+ ((1+x)\rho(X)+x)\norz(X),
\end{equation}
\item there exists a constant $c_0 >0$ such that $\det(d\sigma)$ is uniformly bounded from below by $c_0$,
\item there exists a constant $M$ such that $\Vert d\sigma\Vert_{\mathcal{L}(\R^{d+1},\R^{d+1})}$ is uniformly bounded by $M$.
\end{enumerate}

We say that a domain is admissible when it admits a parametrisation $\rho \in \Ur(\Gamma_0)$.
\end{defn}

\begin{rem}
The condition \eqref{e:diff-bound-val} ensures us in particular that the value of $\sigma$ in a neighborhood of the boundary $\Gamma_0$ does not depends on the choice of $\sigma$.
\end{rem}

\begin{defn} \label{d:adm-param}
We denote by $\varepsilon_0$ the size of the tubular neighborhood $\tub$.
Without loss of generality, we assume that $\varepsilon_0\leq \tfrac{1}{2} < 1$.
We say that a function $\rho\in W^{r,\infty}(\Gamma_0)$ belongs to $\Ur(\Gamma_0)$ when 
\begin{equation}
\Vert \rho\Vert_{L^{\infty}}
    < \min \{\tfrac{\varepsilon_0}{6}, \tfrac{1}{18\Vert \weing\Vert_{L^{\infty}}}\},
\quad \text{ and } \quad
\Vert \rho\Vert_{W^{1,\infty}} 
    < 1.
\end{equation}
\end{defn}

The meaning behind Definition \ref{d:adm-param} is that such a parametrisation allows us to construct admissible diffeomorphism as shown in Proposition \ref{p:adm-diff-constr} below.

\begin{prop} \label{p:adm-diff-constr}
Let $\rho\in\Ur(\Gamma_0)$.
There exists an admissible diffeomorphism $\sigma : \Omega_0\rightarrow \Omega$ of class~$W^{r,\infty}$.
\end{prop}

\begin{proof}
We denote by $\varepsilon_1 := \min\left(\varepsilon_0,\tfrac{1}{3\Vert \weing\Vert_{L^{\infty}}}\right)$.
Let $\delta \in C^{\infty}([-\infty, 0])$ be a non-decreasing cutoff function verifying
\begin{equation}
\forall x\in [-\varepsilon_1/2,0], \, \delta(x) = x+1,
\quad \text{and} \quad
\forall x\in [-\infty, -\varepsilon_1], \, \delta(x) = 0.
\end{equation}
Moreover, we assume that 
\begin{equation*}
\delta' \leq \tfrac{3}{\varepsilon_1}.
\end{equation*}
This function allows us to construct the function $\sigma:\Omega_0\rightarrow\Omega$ through
\begin{subequations} \label{e:syst-diffeo-def}
\begin{align}
\sigma_{|\Omega_0\setminus (\Omega_0\cap \tub)} &= \Id, \\
\forall X+x\norz(X)\in \Omega_0\cap \tub, \quad \sigma(X+x\norz(X)) &= X+ (\delta(x)\rho(X)+x)\norz(X).
\end{align}
\end{subequations}
We  differentiate this expression with respect to $X+x\norz(X)$:
\begin{align*}
d_{X+x\norz(X)}\sigma(H+h\norz(X))
    =& H
       + (\delta(x)\rho(X)+x)d_{X}\norz(H) \\
     & + (\nabla_{\Gamma_0}\rho(X)\cdot H)\norz(X)
       + h(\delta'(x)\rho(X)+1)\norz(X).
\end{align*}
In particular, the determinant of $d\sigma$ is given by:
\begin{equation}
\det(d\sigma)
    = \det\left(\Id+(\delta(x)\rho(X)+x)\weing(X)\right) \, \left( \delta'(x)\rho(X)+1 \right).
\end{equation}
Then, due to the hypothesis on $\varepsilon_1$ and $\rho$, one has, for $X\in\Gamma_0$ and $x\in [-\varepsilon_1,0]$:
\begin{align*}
\Vert (\delta(x)\rho(X)+x)L(X)\Vert
    &\leq (\varepsilon_1+ \Vert\rho\Vert_{L^{\infty}})\Vert L\Vert_{L^{\infty}} \\
    &\leq \tfrac{2}{3} <1,
\end{align*}
and 
\begin{align*}
|\delta'(x)\rho(X)|
    &\leq \tfrac{3}{\varepsilon_1}\Vert\rho\Vert_{L^{\infty}} \\
    &\leq \max\left(\tfrac{3}{\varepsilon_0},9\Vert\weing\Vert_{L^{\infty}}\right)
        \min\left(\tfrac{\varepsilon_0}{6},\tfrac{1}{18\Vert\weing\Vert_{L^{\infty}}}\right) \\
    &\leq \frac{1}{2} <1.
\end{align*}
Therefore for $x\in[-\varepsilon_1,0]$, we have $\det(d\sigma)\geq \tfrac{1}{6}>0$ bounded by below.

Moreover, if we denote by $\underline{\delta}_X$ the function
\begin{equation}
\underline{\delta}_X : x\mapsto \delta(x)\rho(X) +x,
\end{equation}
we just proved that its derivative is positive.
Hence this function is invertible, therefore so is $\sigma$ and in $\sigma(\Omega_0\cap\tub)$, one has:
\begin{equation}
\sigma^{-1}(X+y\norz(X)) = X + \underline{\delta}_X^{-1}(y) \norz(X).
\end{equation}

Since $\rho$ belongs to $W^{1,\infty}$, so does $\sigma$, which ensures that $d\sigma$ is uniformly bounded.
\end{proof}

Due to the fact that we have an explicit formula for admissible diffeomorphisms near $\Gamma_0$, one can compute their gradient and inverse explicitly in this neighbourhood as well.
The proof is straightforward and omitted.

\begin{prop}
Let  $r\in \N$, with $r\geq 2$ and $\rho \in \Ur$.
Let $\sigma$ be an admissible diffeomorphism and $\tubp$ an associated tubular neighborhood.
Let us introduce the application $K$ through:
\begin{equation} \label{e:def-K}
K: X+x\norz(X) \in\tubp \mapsto \mathrm{Id}_{T_X\Gamma_0}+((x+1)\rho(X)+x) \weing(X) \in \mathcal{L}(T_X\Gamma_0).    
\end{equation}

Then one has
\begin{enumerate}
\item[(i)] for $X+y\norz(X)\in \sigma(\tubp)$, one has
\begin{equation} \label{e:sigmaInvForm}
\sigma^{-1}(X+y\norz(X)) 
    = X+\tfrac{y-\rho(X)}{1+\rho(X)}\norz(X),
\end{equation}
\item[(ii)] for $X+x\norz(X)\in\tubp$, one can now write $d\sigma$ matricially by block
\begin{equation}
d_{X+x\norz(X)}\sigma 
    = 
    \begin{pmatrix}
      K                               & 0 \\
      (x+1)\,^t\nabla_{\Gamma_0}\rho  & 1+\rho
    \end{pmatrix},
\end{equation}
in the decomposition $\R^{d+1}= T_{X}\Gamma_0 \oplus \R \norz(X)$,
\item[(iii)] 
the linear application $K$ is invertible and one gets:
\begin{equation}
(d_{X+x\norz(X)}\sigma)^{-1}(H+h\norz) 
    = K^{-1}H
     - \tfrac{x+1}{1+\rho} \left(\nabla_{\Gamma_0}\rho \cdot K^{-1}H\right)\norz
     + \tfrac{1}{1+\rho} h\norz.
\end{equation}

One can also compute the transposition $\,^t(d\sigma)^{-1}$
\begin{equation} \label{e:FormulaDSigInv}
\,^t(d_{X+x\norz(X)}\sigma)^{-1} = \,^tK^{-1}H + \tfrac{h}{1+\rho}\left(\norz-(x+1)\,^tK^{-1}\nabla_{\Gamma_0}\rho\right).
\end{equation}
\end{enumerate}
\end{prop}

\subsection{Definition of the Dirichlet-to-Neumann operator} \label{ss:def-DtN}


The Dirichlet to Neumann operator on a domain $\omega$ is the operator sending a function $\psi \in H^{1/2}(\partial\omega)$ on to the normal derivative of the harmonic extension $\psi^{\mathrm{h}}$ of $\psi$:
\begin{subequations}
\begin{align}
\Delta \psi^{\mathrm{h}} &= 0, \quad \text{on } \omega, \\
\psi^{\mathrm{h}} &= \psi, \quad \text{on } \partial\omega,
\end{align}
\end{subequations}
and 
\begin{equation}
\G[\omega](\psi) := \partial_{\mathbf{n}} \psi^{\mathrm{h}}. 
\end{equation}
In this paragraph, we give a slightly different definition.
The operator that we consider, denoted $\G[\rho]$, will be more or less the pull back of $\G[\Omega(\rho)]$ on to the boundary $\Gamma_0$ of $\Omega_0$. \\

Let $r\in \mathbb{N}_{\geq 2}$ and $\rho\in U^r(\Gamma_0)$.
We fix a diffeomorphism $\sigma$ associated with $\rho$.

Let $f:\Omega_0\rightarrow \R$.
We introduce the transported gradient of $f$ via $\sigma$, denoted $\nabs$, by:
\begin{align*}
\nabs f 
    &:= \nabla (f\circ \sigma^{-1}) \circ \sigma \\ 
    &=\, ^t(d\sigma)^{-1} \, \nabla f            \\
    &=\, ^tK^{-1}\nabla_{\Gamma_0}f + \tfrac{\norz\cdot \nabla f}{1+\rho} (\norz - (x+1)\,^tK^{-1}\nabla_X \rho).
\end{align*}

We introduce $P(\sigma)$ the following matrix:
\begin{equation}
P(\sigma) := |\det(d\sigma)|\, (d\sigma)^{-1}\, ^t(d\sigma)^{-1}.
\end{equation}
Let us remark that for $\rho \in W^{r+1,\infty}(\Gamma_0)$, the matrix $P$ is in $W^{r,\infty}(\Gamma_0)$.
Moreover, for $\rho \in U^r(\Gamma_0)$, the matrix $P(\sigma)$ is uniformly coercive on $\Omega_0$.

Let $\psi\in H^{\tfrac{3}{2}}(\Gamma_0)$.
We introduce $\phi = \phi[\sigma,\psi]$ the solution of the following elliptic problem:
\begin{align}
\div(P(\sigma)\nabla \phi) &= 0  \quad \text{on } \Omega_0, \\
\phi &= \psi                     \quad \text{on } \Gamma_0.
\end{align}
One can check that $\phi \in H^2(\Omega_0)$.
Moreover, if we introduce the function $\Phi:= \phi\circ\sigma^{-1} \in H^2(\Omega)$, then $\Phi$ is the solution of the following problem:
\begin{align}
\Delta\Phi &= 0             \quad \text{on } \Omega, \\
\Phi &= \psi\circ\sigma^{-1}     \quad \text{on } \Gamma.
\end{align}
In particular, for $\rho$ fixed, the function $\Phi_{|\sigma(\tubp\cup\Omega_0)}$ does not depend on the choice of $\sigma$.
A useful consequence of this fact, using \eqref{e:sigmaInvForm}, is that the restriction of $\phi$ to $\tubp\cup\Omega_0$ does not depend on the choice of $\sigma$.
When this does not lead to confusion, we write $\phi[\rho,\psi]$ instead of $\phi[\sigma,\psi]$.

\begin{defn}
The operator $\G[\rho]$, defined through
\begin{equation}
\G[\rho] : 
\psi \in H^{\tfrac{3}{2}}(\Gamma_0)
\mapsto 
\left(\tfrac{1}{\det(K)}P(\sigma)\nabla\phi \cdot \norz\right)_{|\Gamma_0} \in H^{\tfrac{1}{2}}(\Gamma_0),
\end{equation}
is called the \textit{Dirichlet-to-Neumann} operator.
\end{defn}

The Dirichlet-to-Neumann operator is an operator of order $1$ in the sense of the following proposition.
For a proof, we refer to \cite{Mclean}, Theorem 4.21, on the \textit{Poincaré-Steklov operator}, which is an other name for Dirichlet-to-Neumann.

\begin{prop}
For $s\geq 0$ a real number, $r\geq s$ an integer  and $\rho \in U^{r+1}(\Gamma_0)$, the operator $\G[\rho]$ can be extended from $H^{s+1/2}(\Gamma_0)$ to $H^{s-1/2}(\Gamma_0)$.
\end{prop}

\begin{rem}
Introducing the vector field $\Nor:\Gamma_0\mapsto \R^{d+1}$ as follows,
\begin{align} \notag
\Nor 
    &:= (1+\rho)\,^t(d\sigma)^{-1}\norz, \\ \label{e:unit-vect}
    &= \norz - \,^tK^{-1}\nabla_{\Gamma_0} \rho,
\end{align}
one can rewrite the Dirichlet-to-Neumann operators as 
\begin{equation}
\G[\rho]\psi = \nabs \phi \cdot \Nor.
\end{equation}
Hence, this definition of $\G$ is not exactly the same as the classical one, since $\Nor$ is normal to $\Gamma$, but not unitary.
It differs from the classical definition through the multiplication by $1/|N|$, which is a function depending on $\rho$.
\end{rem}

\subsection{Shape derivative of the Dirichlet-to-Neumann operator} \label{ss:shape-der-DtN}

The classical setting for \textit{shape derivatives} is to introduce, for $\theta$ a diffeomorphism, the set $\Omega_{\theta} := (\Id + \theta)(\Omega)$.
Then one differentiates with respect to $\theta$ the functions and functionals that depend on $\theta$.
This can be difficult to implement, because the variables with respect to which we differentiate are diffeomorphisms, and more difficult to compute during numerical simulations.

We take a slightly different approach here, made possible by the hypothesis that $\Gamma$ is parameterized over $\Gamma_0$, and look at Fréchet derivatives of our functions with respect to $\rho$:
\begin{equation}
f[\rho+\varepsilon\rhop]
    \underset{\varepsilon\rightarrow 0}{=} f[\rho] 
     + \varepsilon d_{\rho}f(\rhop)
     + o(\varepsilon).
\end{equation}
Both approaches are valid and we refer to the classic \cite{ShapeDerLivre} for more information on shape derivatives.

Our goal in this paragraph is to derive an explicit formula for the shape derivative of the Dirichlet-to-Neumann operator.
Let us start with the differentiability.

\begin{prop} \label{p:shape-derivability}
Let $s\geq 0$ a real number, $r\geq s$ an integer.
The operator $\G$, viewed as a function from $U^{r+1}(\Gamma_0)$ to $ \mathcal{L}\left(H^{s+1/2}(\Gamma_0),H^{s-1/2}(\Gamma_0)\right)$
is continuous and Fréchet differentiable.
\end{prop}

\begin{proof}
The idea behind this proof is to look at the solution $\phi$ and its perturbation when $\rho$ changes. 
This perturbation solves an elliptic equation, on which we can perform estimates to prove that it is small.

For $\rho,\rhop\in W^{r+1,\infty}(\Gamma_0)$ and $\psi \in H^{s+1/2}(\Gamma_0)$, let us introduce $\tilde{\phi}[\rho,\rhop,\psi]$ the solution of the elliptic problem:
\begin{align}
\div(P(\sigma)\nabla \tilde{\phi}) &= -\div(d_{\sigma}P(\sigmap)\nabla \phi[\rho,\psi])  \quad \text{on } \Omega_0, \\
\tilde{\phi} &= 0           \quad \text{on } \Gamma_0,
\end{align}
where $d_{\sigma}P$ is the derivative of the matrix $P$ with respect to $\sigma$ and $\sigmap$ is the change of $\sigma$ with respect to~$\rhop$:
\begin{equation}
\sigmap(z) 
    := \begin{cases}
           0 & \text{ if } z \notin \tub, \\
           \delta(x) \rhop(X)\norz(X) & \text{ if } z= X+x\norz(X) \in \tub.
       \end{cases}
\end{equation}
As $\rhop \in W^{r+1,\infty}(\Gamma_0)$, we have $d_{\sigma}P(\sigmap)\in W^{r,\infty}(\Omega_0)$, and then $\div(d_{\sigma}P(\sigmap)\nabla \phi[\rho,\psi])\in H^{s-1}(\Omega_0)$.
Then, due to elliptic regularity, we can deduce that $\tilde{\phi} \in H^{s+1}(\Omega_0)$.

Now, we introduce $R_0 = R_0(\rho,\rhop,\psi,\varepsilon)$ and $R_1= R_1(\rho,\rhop,\psi,\varepsilon)$ the Taylor remainders of the solution $\phi$ at order $0$ and $1$, with respect to $\rho$:
\begin{subequations}
\begin{align}
R_0 &:= \phi[\rho+\varepsilon\rhop,\psi]- \phi[\rho,\psi], \\
R_1 &:= \phi[\rho+\varepsilon\rhop,\psi]- \phi[\rho,\psi] - \varepsilon \tilde{\phi}[\rho,\rhop,\psi].
\end{align}
\end{subequations}
We also introduce $P_0= P_0(\rho,\rhop,\varepsilon)$ and $P_0= P_0(\rho,\rhop,\varepsilon)$ the Taylor remainders of the matrix $P$ with respect to~$\rho$:
\begin{subequations}
\begin{align}
P_0 &:= P(\sigma+\varepsilon\sigmap)-P(\sigma), \\
P_1 &:= P(\sigma+\varepsilon\sigmap)-P(\sigma)- \varepsilon d_{\sigma}P(\sigmap).
\end{align}
\end{subequations}
We have $(R_0)_{|\Gamma_0}=(R_1)_{|\Gamma_0}= 0$, and inside the domain $\Omega_0$, one has:
\begin{subequations}
\begin{align}
\div(P(\sigma)\nabla R_0) 
    &= \div\big( P_0\nabla \phi[\rho+\varepsilon\rhop,\psi] \big), \\
\div(P(\sigma)\nabla R_1) 
    &= \, \div\big( P_1 \nabla \phi[\rho+\varepsilon\rhop,\psi] \big) 
     - \varepsilon \div\big(d_{\sigma}P(\sigmap)\nabla R_0 \big).
\end{align}
\end{subequations}
By elliptic regularity, we get the bounds:
\begin{subequations}
\begin{align}
\Vert R_0\Vert_{H^{s+1}} 
    &\lesssim \Vert P_0\Vert_{W^{r,\infty}} \Vert \phi[\rho+\varepsilon\rhop,\psi]\Vert_{H^{s+1}}, \\
\Vert R_1\Vert_{H^{s+1}} 
   &\lesssim \Vert P_1\Vert_{W^{r,\infty}} \Vert \phi[\rho+\varepsilon\rhop,\psi]\Vert_{H^{s+1}}
    + \varepsilon \Vert d_{\sigma}P(\sigmap)\Vert_{W^{r,\infty}} \Vert R_0\Vert_{H^{s+1}},
\end{align}
\end{subequations}
with constants depending only on the domain $\Omega_0$.
Then, using the fact that $\Vert P_0\Vert_{W^{r,\infty}} \underset{\varepsilon\rightarrow 0}{=} o(1)$ and $\Vert P_1\Vert_{W^{r,\infty}} \underset{\varepsilon\rightarrow 0}{=} o(\varepsilon)$, we get that $\Vert R_0\Vert_{H^{s+1}} \underset{\varepsilon\rightarrow 0}{=} o(1)$, granting us the continuity of $\G$ and then $\Vert R_1\Vert_{H^{s+1}} \underset{\varepsilon\rightarrow 0}{=} o(\varepsilon)$, granting us the differentiability of $\G$.
\end{proof}

\begin{rem}
Such a proof can be easily extended to prove the analiticity of the Dirichlet-to-Neumann operator with respect to shape (see \cite{Lannes} for the case of the strip for example).
\end{rem}

Now that we know that $\G$ is differentiable with respect to shape,
let us compute its derivative explicitly.

\begin{thm} \label{t:shape-der-formula}
The shape derivative of the Dirichlet-to-Neumann operator is given through the following formula:
\begin{equation} \label{e:shape-der-form}
\mathrm{d}_{\rho}\G(\rhop)\psi 
    = -\G[\rho](\rhop w) 
    - \tfrac{1}{\det(K)}\div_{\Gamma_0}(\rhop \mathbf{V})
    + \rhop\, a_0(\rho,\psi),
\end{equation}
where $w$, $\mathbf{V}$ and $a_0$ are given through:
\begin{subequations}
\begin{align}
w          
    &:= \frac{\partial_{\norz}\phi}{1+\rho} 
    = \frac{\G\psi + G\nabla_{\Gamma_0}\rho\cdot \nabla_{\Gamma_0}\psi}{1+|\,^tK^{-1}\nabla_{\Gamma_0}\rho|^2}, \\
\mathbf{V} 
    &:= \det(K)K^{-1}\Pi_0\nabs\phi
    = \det(K)G(\nabla_{\Gamma_0}\psi-w\nabla_{\Gamma_0}\rho), \\
a_0(\rho, \psi)
    &:= - \left(\frac{ H_0}{1+\rho}+\mathrm{tr}(K^{-1}d\norz)\right)\G\psi
     + \frac{1}{\det(K)}\mathbf{V}\cdot d\norz\,^{t}K^{-1}\nabla_{\Gamma_0}\rho, \\
G
    &:= K^{-1}\,^tK^{-1}.
\end{align}
\end{subequations}
\end{thm}

\begin{rem}
By default, when it is ommitted, $\G\psi$ means $\G[\rho](\psi)$.
\end{rem}


\begin{proof}[Proof of Theorem \ref{t:shape-der-formula}]
The primal idea behind this computation is to introduce \textit{Alinhac's good unknown}.
For $f$ a function depending on the variable $Y=X+x\norz$ and on $\rho$, which we wish to differentiate with respect to $\rho$, we introduce
\begin{align*}
\mathrm{d}_{\rho}^{\sigma}f(\rhop)(Y)
    &:= \mathrm{d}_{\rho}\left(f\circ\sigma^{-1}\right)(\rhop)(\sigma(Y)) \\
    &= \mathrm{d}_{\rho}f(\rhop) + d_{Y}f(\mathrm{d}_{\rho}(\sigma^{-1})(r)\circ\sigma)  \\
    &= \mathrm{d}_{\rho}f(\rhop) - \tfrac{\rhop}{1+\rho}\partial_{\norz}f.
\end{align*}

The main point of the good unknown is Lemma \ref{l:commmutation-alinhac-transnabla}:

\begin{lem} \label{l:commmutation-alinhac-transnabla}
Let $f : C^1(W^{2,\infty}(\Gamma_0)\times \Omega_0, \R)$ be a function depending on the shape of the domain and the space variable, then one has the following commutation formula:
\begin{equation}
\mathrm{d}_{\rho}^{\sigma}\nabs f(\rhop) = \nabs\mathrm{d}_{\rho}^{\sigma}f(\rhop).
\end{equation}
\end{lem}

\begin{proof}[Proof of Lemma \ref{l:commmutation-alinhac-transnabla}]
We are just saying that derivatives commutes in the moving space $\Omega$.
Let us introduce the function $F$ defined by 
$$ F(\rho, z) := f(\rho,\sigma^{-1}(z)). $$
The formula above corresponds to 
\begin{equation}
\mathrm{d}_{\rho}\nabla F(\rhop)(z) = \nabla\mathrm{d}_{\rho}F(\rhop)(z).
\end{equation}
\end{proof}

\begin{rem}
In this context, Alinhac's good unknown is just a change of coordinates.
This change of coordinates, as well as its utility in the context of hyperbolic system was historically introduced by Alinhac in \cite{Alinhac1} and later on in \cite{Alinhac2}.
The idea behind is that this derivatives are more regular than the classical ones, because some of the regularity falls on the diffeomorphism.
You can find a more in depth discussion in the case of hyperbolic free boundary problem in \cite{Iguchi-Lannes1}.

At low regularity, Alinhac's good unknown is more complex, because the product $\tfrac{\rhop}{1+\rho}\partial_{\norz}f$ on the second term is not well defined.
In this case, one uses instead the lower frequency term of the associated para-product decomposition.
See \cite{ParaLinDtN} 
for a more detailed computation using the aforementioned technics.
\end{rem}

We wish to decompose $\div(P(\sigma)\nabla\phi)$ in terms of tangential and normal derivatives.

\begin{lem}\label{l:div-decomp-tang-norm}
The \textit{transported laplacian} $\div(P(\sigma)\nabla\phi)$, once restricted to the boundary $\Gamma_0$, can be decomposed in terms of  tangential and  normal components as follows
\begin{equation} \label{e:decomp-lap-transp}
\div(P(\sigma)\nabla\phi) 
    = (1+\rho)\div_{\Gamma_0}\left(\det(K)K^{-1}\Pi_0\nabs\phi\right)  
     + \left(H_0+\partial_{\norz}\right)\left(\det(K) \Nor\cdot \nabs\phi\right),
\end{equation}
where $\Nor(X+x\norz(X)) := \Nor(X)$ is considered as constant with respect to the normal coordinate.
\end{lem}

\begin{rem}
In particular, for $\rho = 0$, we recover the classical decomposition of the laplacian along tangential and normal derivative using the curvature of $\Gamma_0$:
\begin{equation}
\Delta \phi 
    = \div_{\Gamma_0}(\nabla_{\Gamma_0} \phi)
     + (H_0+\partial_{\norz}) \partial_{\norz} \phi.
\end{equation}
This formula is true in all of $\tub$. 
\end{rem}

\begin{proof}[Proof of Lemma \ref{l:div-decomp-tang-norm}]
We decompose the divergence operator taken at the boundary $\Gamma_0$.
Let $\mathbf{v} \in C^1(\tub,\R^{d+1})$ and $X\in \Gamma_0$.
Since $\Pi_0 \mathbf{v}(X+x\norz(X))\cdot \norz(X)=0$ is constant with respect to $x$, one has 
\begin{equation*}
d_{X+x\norz(X)}\mathbf{v}(\norz(X))\cdot\norz(X) = 0.
\end{equation*}
In particular, we have 
\begin{equation} \label{e:div-tang}
\div(\Pi_0 \mathbf{v}) = \div_{\Gamma_0}(\Pi_0 \mathbf{v}).
\end{equation}
Moreover, we have
\begin{align*}
\div(\Pi_0^{\perp}\mathbf{v})
    &= \div\left((\mathbf{v}\cdot\norz)\norz\right) \\ 
    &= \nabla(\mathbf{v}\cdot\norz) \cdot \norz 
     + (\mathbf{v}\cdot\norz) \div(\norz).
\end{align*}
Then, with, the assumption that $\norz$ is constant alongside $x$, namely $\norz(X+x\norz(X)) = \norz(X)$, we get
\begin{equation}\label{e:div-norm}
\div(\Pi_0^{\perp}\mathbf{v})
    = (H_0 + \partial_{\norz})(\mathbf{v}\cdot \norz).
\end{equation}

Combining \eqref{e:div-tang} and \eqref{e:div-norm}, we get:
\begin{align} \notag
\div (\mathbf{v}) 
    &= \div (\Pi_0(\mathbf{v}) + \Pi_0^{\perp}(\mathbf{v})) \\ \label{e:div-tan-norm}
    &= \div_{\Gamma_0}(\Pi_0 \mathbf{v}) 
     + (H_0 + \partial_{\norz})(\mathbf{v}\cdot \norz).
\end{align}
We decompose $P(\sigma)\nabla\phi$ under $\R^{d+1} = T_X\Gamma_0 \oplus T_X^{\perp}\Gamma_0$ for $X+x\norz(X) \in \tubp$ close enough of $\Gamma_0$:
\begin{align*}
P(\sigma)\nabla\phi
    &= \det(\mathrm{d}\sigma) (\mathrm{d}\sigma)^{-1} \,^t(\mathrm{d}\sigma)^{-1} \nabla\phi \\
    &= \det(\mathrm{d}\sigma) (\mathrm{d}\sigma)^{-1} \, \nabs \phi \\
    &= \det(\mathrm{d}\sigma) (\mathrm{d}\sigma)^{-1}(\Pi_0\nabs\phi+\Pi_0^{\perp}\nabs\phi) \\
    &= \det(\mathrm{d}\sigma) K^{-1}\Pi_0\nabs\phi 
        + \det(K)\left(\norz\cdot\nabs\phi-(x+1)\nabla\rho\cdot K^{-1}\Pi_0\nabs\phi\right)\norz \\
    &= \det(\mathrm{d}\sigma) K^{-1}\Pi_0\nabs\phi  
        + \det(K)(\Nor\cdot \nabs\phi)\norz
        - x\det(K) (\nabla_{\Gamma_0}\rho\cdot K^{-1}\Pi_0\nabs\phi)\norz.
\end{align*}
We used the definition of $P(\sigma)$ for the first equality, the definition of $\nabs$ for the second.
We decomposed along $\Id = \Pi_0 + \Pi_0^{\perp}$ for the third.
The fourth is an application of the formula \eqref{e:FormulaDSigInv} for $(d\sigma)^{-1}$.
And the last equality is just a rewriting of term in a way to isolate what will give the Dirichlet-to-Neumann operator.

Then, we use formula \eqref{e:div-tan-norm} and get for $x=0$
\begin{align*} 
\div(P(\sigma)\nabla\phi) 
    =& \, \div_{\Gamma_0}\left((1+\rho)\det(K)K^{-1}\Pi_0\nabs\phi\right) \\ 
     &+ \left(H_0+\partial_{\norz}\right)( \det(K) \Nor\cdot \nabs\phi)
     - \det(K) (\nabla_{\Gamma_0}\rho\cdot K^{-1}\Pi_0\nabs\phi) \\  
    =& \, (1+\rho)\div_{\Gamma_0}\left(\det(K)K^{-1}\Pi_0\nabs\phi\right)  
     + \left(H_0+\partial_{\norz}\right)(\det(K) \Nor\cdot \nabs\phi).
\end{align*}
In the first equality, we inject the formula for $P(\sigma)\nabla \phi$ in the divergence, and make the simplification due to the presence of $x$ at $x=0$, remark that there is one term remaining due to the $\partial_{\norz}$ (i.e. partial derivative with respect to $x$).
In the second equality, we cancel the first and third terms to commute $1+\rho$ with $\div_{\Gamma_0}$. 
\end{proof}

Let us resume with the proof of Theorem \ref{t:shape-der-formula}.
Now that we have this decomposition and Alinhac's good unknown in mind let us compute the shape derivative of the Dirichlet to Neumann operator.
Expressing it as a product, we can separate the contribution of $\rhop$ to the boundary from the one to the potential function $\phi$:
\begin{equation} \label{e:DtN-shape-leibniz}
\mathrm{d}_{\rho}\G(\rhop)\psi 
    = \mathrm{d}_{\rho}\nabs\phi(\rhop)\cdot\Nor 
     + \nabs\phi\cdot \mathrm{d}_{\rho}\Nor(\rhop).
\end{equation}

Then each part can be simplified as follows
\begin{itemize}
\item 
To simplify the first term, we use Alinhac's good unknown
\begin{equation} \label{chap4-e:deb-calc}
\mathrm{d}_{\rho}\nabs\phi(\rhop) 
    = \mathrm{d}_{\rho}^{\sigma}\nabs\phi(\rhop) 
     + \tfrac{\rhop}{1+\rho}\partial_{\norz}\nabs\phi.
\end{equation}
Then, using the Lemma \ref{l:commmutation-alinhac-transnabla}, as well as the definition of the Dirichlet-to-Neumann operator, we get:
\begin{align*}
\mathrm{d}_{\rho}^{\sigma}\nabs\phi(\rhop) \cdot \Nor
    &= \nabs \mathrm{d}_{\rho}^{\sigma}\phi(\rhop) \cdot \Nor \\
    &= \G (\mathrm{d}_{\rho}^{\sigma}\phi(\rhop)).
\end{align*}
With the definition of $\mathrm{d}_{\rho}^{\sigma}$, and the fact that $\phi$ is fixed on $\Gamma_0$, one gets
\begin{equation}
\mathrm{d}_{\rho}^{\sigma}\phi(\rhop) 
    = - \tfrac{\rhop}{1+\rho} \partial_{\norz}\phi,
\end{equation}
which allows us to simplify the first term of \eqref{chap4-e:deb-calc}

We now focus on the second one, which we wish to transform into the last term of \eqref{e:decomp-lap-transp}.
We remark that 
\begin{equation}
\partial_{\norz}(\det(K)) 
    = \det(K) \mathrm{tr}(K^{-1}\partial_{\norz}K)
    = \det(K)(1+\rho) \mathrm{tr}(K^{-1}d\norz),
\end{equation}
and therefore
\begin{equation}
\partial_{\norz}(\det(K)\nabs\phi\cdot \Nor)
    = \det(K) \partial_{\norz}\nabs\phi\cdot \Nor
     + \det(K)(1+\rho) \mathrm{tr}(K^{-1}d\norz)\nabs\phi\cdot \Nor.
\end{equation}

Finally
\begin{equation} \label{e:div-dir-t-neu-part1}
\mathrm{d}_{\rho}\nabs\phi(\rhop) \cdot \Nor 
    = - \G\left(\tfrac{\rhop}{1+\rho} \partial_{\norz}\phi\right) 
     + \tfrac{\rhop}{\det(K)(1+\rho)}\partial_{\norz}\left( det(K)\nabs\phi\cdot \Nor\right)
     - \rhop \mathrm{tr}(K^{-1}d\norz)\G\psi.
\end{equation}

\item 
For the derivative of $\Nor$, we get
\begin{equation} \label{e:div-dir-t-neu-part2}
\mathrm{d}_{\rho}\Nor(\rhop) 
    = - \,^tK^{-1}\nabla_{\Gamma_0} \rhop 
     - \mathrm{d}_{\rho}(\,^tK^{-1})(\rhop)\, \nabla_{\Gamma_0}\rho.
\end{equation}

\begin{itemize}
\item 
To simplify the first term of \eqref{e:div-dir-t-neu-part2}, we use \eqref{e:divg-prod},
\begin{equation} \label{e:div-dir-t-neu-part3}
\det(K)\nabla_{\Gamma_0}\rhop \cdot K^{-1}\Pi_0\nabs\phi
    = \div_{\Gamma_0}(\rhop \det(K) K^{-1}\Pi_0\nabs\phi)
    - \rhop\, \div_{\Gamma_0}(\det(K) K^{-1}\Pi_0\nabs\phi).
\end{equation}
The second term of the right hand side of \eqref{e:div-dir-t-neu-part3} is already in $\tfrac{\rhop}{1+\rho}\div(P(\sigma)\nabla\phi)$, which will allow us to cancel it later on.

\item
Let us now simplify the second term of \eqref{e:div-dir-t-neu-part2}.
Using the definition \eqref{e:def-K} of $K$, one gets, for $x=0$:
\begin{equation}
d_{\rho}(K)(\rhop)_{|\Gamma_0} = \rhop \, d_X\norz.
\end{equation}
Then, we call the application $\theta:M\in \mathcal{GL}T_{X}\Gamma_0 \mapsto M^{-1}\in \mathcal{GL}T_{X}\Gamma_0$.
One gets 
\begin{equation}
d_{M}\theta\big(\dot{M}\big) = -M^{-1}\dot{M}M^{-1},
\end{equation}
and by applying the chain rule:
\begin{equation}\label{e:der-autre}
\mathrm{d}_{\rho}(\,^tK^{-1})(\rhop)\, \nabla_{\Gamma_0}\rho 
    =  -\rhop \,^tK^{-1}d\norz \,^tK^{-1}\nabla_{\Gamma_0}\rho.
\end{equation}
\end{itemize}
\end{itemize}

Combining \eqref{e:div-dir-t-neu-part1}, \eqref{e:div-dir-t-neu-part2}, \eqref{e:div-dir-t-neu-part3} and \eqref{e:der-autre}, one gets 
\begin{align*}
\mathrm{d}_{\rho}\G(\rhop)\psi 
    =&- \G\left(\tfrac{\rhop}{1+\rho} \partial_{\norz}\phi\right) 
     + \tfrac{\rhop}{\det(K)(1+\rho)}\partial_{\norz}\left( det(K)\nabs\phi\cdot \Nor\right)
     - \rhop \mathrm{tr}(K^{-1}d\norz)\G\psi.\\
    &- \tfrac{1}{\det(K)}\div_{\Gamma_0}(\rhop \det(K) K^{-1}\Pi_0\nabs\phi)
     + \tfrac{\rhop}{\det(K)} \div_{\Gamma_0}(\det(K) K^{-1}\Pi_0\nabs\phi) \\
    &+ \rhop \,^tK^{-1}d\norz \,^tK^{-1}\nabla_{\Gamma_0}\rho\cdot\nabs\phi \\
    =&- \G(\tfrac{\rhop}{1+\rho} \partial_{\norz}\phi) 
      - \tfrac{1}{\det(K)}\div_{\Gamma_0}(\rhop \det(K)K^{-1}\Pi_0\nabs\phi) \\
    &+ \rhop \,^tK^{-1}d\norz \,^tK^{-1}\nabla_{\Gamma_0}\rho\cdot\nabs\phi
      - H_0 \tfrac{\rhop}{1+\rho} \Nor\cdot\nabs\phi
      - \rhop \mathrm{tr}(K^{-1}d\norz)\G\psi\\
    &+ \tfrac{\rhop}{(1+\rho)\det(K)} \left(
           (1+\rho)\div_{\Gamma_0}\left(\det(K)K^{-1}\Pi_0\nabs\phi\right)  
          + \left(H_0+\partial_{\norz}\right)\left(\det(K)\Nor\cdot \nabs\phi\right)\right).
\end{align*}
Then, using \eqref{e:decomp-lap-transp} as well as the fact that $\div(P(\sigma)\nabla\phi)=0$, we get:
\begin{align} \notag
\mathrm{d}_{\rho}\G(\rhop)\psi  
    =&- \G(\tfrac{\rhop}{1+\rho} \partial_{\norz}\phi) 
      - \tfrac{1}{\det(K)}\div_{\Gamma_0}(\rhop \det(K)K^{-1}\Pi_0\nabs\phi) \\
     &+ \rhop \,^tK^{-1}d\norz \,^tK^{-1}\nabla_{\Gamma_0}\rho\cdot\nabs\phi
      - H_0 \tfrac{\rhop}{1+\rho} \Nor\cdot\nabs\phi
      - \rhop tr(K^{-1}L_X)\G\psi.
\end{align}
\end{proof}
%


\section{Linearisation of the equation} \label{s:Lin-Equa}


\subsection{Reminder on the biological model} \label{ss:remind-bio}

We are given a cell, inside an ambient liquid.
We denote $\Omega^{\mathrm{tot}}$ the reference domain for the total experience (in this Chapter it is considered fixed, but it could also vary over time).
We fix $\Omega^i_0$ the reference domain for the interior of the cell and $\Omega^i(t)$ the actual cell at time~$t$.
We denote respectively
\begin{equation}
\Omega^{e}_0:= \Omega^{\mathrm{tot}}\setminus \Omega^i_{0}
\quad \text{and} \quad
\Omega^e(t) := \Omega^{\mathrm{tot}}\setminus \Omega^i(t),
\end{equation}
the reference and actual domain for the ambiant fluid.
We denote by $\Gamma_0 := \partial\Omega^i_0$ the reference interface between the cell and the ambiant liquid and by $\Gamma(t) := \partial\Omega^i(t)$ the actual boundary of the cell.
We denote by $\Gamma^{\mathrm{ext}} := \partial\Omega^{\mathrm{tot}} = \partial\Omega^e_0 \setminus \Gamma_0$ the outer boundary of the  domain.
Let us also recall that $\norz^e(X)$, the outward unit normal vector to the outside domain $\Omega^e_0$ at a point $X\in \Gamma_0$, is opposite to the one to $\Omega^{in}_{0}$:
\begin{equation*}
\norz^e := -\norz.
\end{equation*}

\begin{figure} \centering  
\includegraphics[scale= 0.5]{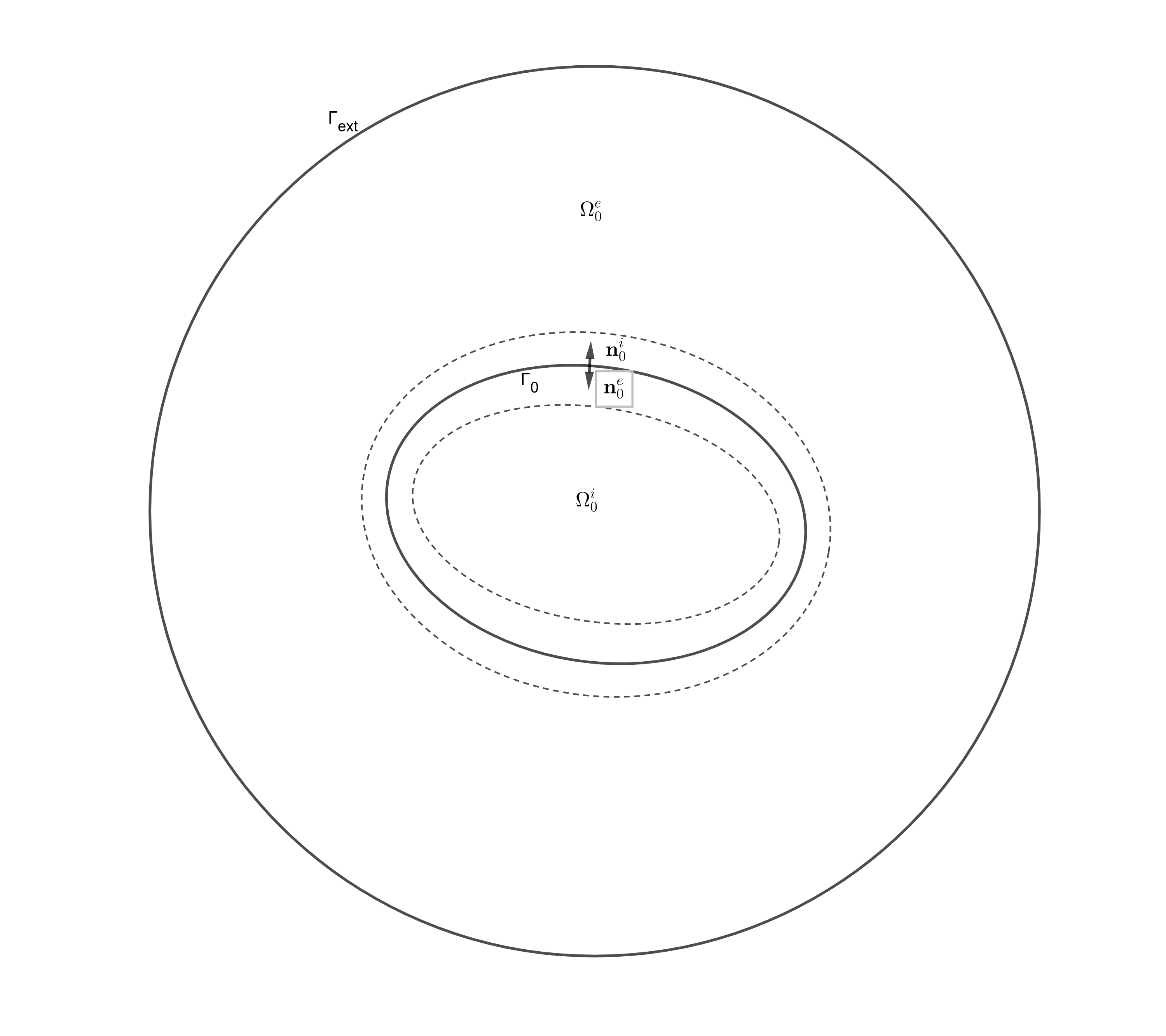}
 \caption{Interior and exterior domains} \label{fig2} 
 \end{figure}

We make the hypothesis that the deformation of the cell remains small over time.
By that, we mean $\Gamma$ is parametrisable over $\Gamma_0$ for all times considered here :
\begin{equation*}
\Gamma(t) = \{X+\rho(X,t)\norz(X) \,;\, X\in \Gamma_0\}.
\end{equation*}
The main unknown of our problem is the shape of the cell, or its parametrisation $\rho$.
With this parametrization, we recall that the unit vector normal to the surface $\Gamma$ is given by:
\begin{equation}
\nor(X+\rho(X,t)\norz(X)) := \tfrac{\Nor}{|\Nor|},
\end{equation}
where $\Nor := \Nor(\rho,X,t)$ is the vector defined in \eqref{e:unit-vect} that we used to define the Dirichlet-to-Neumann operator. 
We recall that $|\Nor|$ is given by:
\begin{equation}
|\Nor| =  \sqrt{1+\left|\,^tK^{-1}\nabla_{\Gamma_0}\rho\right|^2}.
\end{equation}

We recall that the main data of the problem are :
\begin{itemize}
\item the initial shape of the cell $\rho_0(X)$,
\item the concentration of (MT1-MMPs) at the boundary $\Psi$.
\end{itemize}

The phenomenae at stakes are the diffusion of the degraded ligant outside the cell,
\begin{subequations}
\begin{align}
\Delta \Phi^e &= 0, \quad \text{on } \Omega^e, \\
\partial_{\nor^e}\Phi^e=-\partial_{\nor}\Phi^e &= \Psi, \quad \text{on } \Gamma, \\
\Phi^e &= 0, \quad \text{on } \Gamma^{\mathrm{ext}};
\end{align}
\end{subequations}
and the signal inside the cell corresponding to aforementioned diffusion,
\begin{subequations}
\begin{align}
\Delta \Phi^i &= 0, \quad \text{on } \Omega^i, \\
\Phi^i_{|\Gamma(t)} &= \Phi^e_{|\Gamma(t)}, \quad \text{on } \Gamma.
\end{align}
\end{subequations}
At the end the cell moves due to the polymerisation of actin. 
We model this phenomenon by assuming  that a point $\alpha = X + \rho(X,t)\norz(X) \in \Gamma(t)$ on the boundary moves with speed $v$ given by
\begin{equation} \label{e:vitesse-signal}
v(\alpha, t) = \nabla \Phi^i[\rho(t),\Psi(t)] (\alpha). 
\end{equation}

\subsection{Derivation of the evolution equation for the boundary.} \label{ss:evol-equa}

We denote by $\sigma^i: \Omega_0^i \rightarrow \Omega^i$ and $\sigma^e: \Omega_0^e \rightarrow \Omega^e$ the two diffeomorphisms associated with $\Omega^i$, and $\Omega^e$ through \eqref{e:syst-diffeo-def}.
We define $\sigma : \Omega^{\mathrm{tot}}\rightarrow \Omega^{\mathrm{tot}}$, through 
\begin{equation}
\sigma_{|\Omega_0^i} = \sigma^i,
\quad \text{and} \quad
\sigma_{|\Omega_0^e} = \sigma^e.
\end{equation}
Remark that $\sigma$ doesn't have a jump at the boundary $\Gamma_0$ due to the fact that $\sigma^i$ and $\sigma^e$ have the same expression in a neighborhood of the boundary.

Let $\phi^i := \Phi^i\circ \sigma^i$ and $\phi^e := \Phi^e\circ \sigma^e$ and $\psi(t,X) := \Psi(X+\rho(X,t)\norz(X),t)$.
By definition of the Dirichlet-to-Neumann operator, we have:
\begin{equation} \label{e:psi-phie}
\G^e[-\rho](\phi^e_{|\Gamma_0}) = |\Nor|\psi,
\end{equation}
where $\G^e$ is the Dirichlet-to-Neumann operator on $\Omega^e$ with homogeneous boundary condition on $\Gamma^{\mathrm{ext}}$.
Remark that since $\norz^e = -\norz$, we have to take the parametrisation $-\rho$ instead of $\rho$ in $\G^e$ in order to describe the boundary $\Gamma$.

\begin{defn}
Let $\rho \in W^{2,\infty}(\Gamma_0)$.
For $\psi\in H^{\tfrac{1}{2}}(\Gamma_0)$,
we introduce $\phi^{e}[\rho,\psi]$ the solution of the Neumann elliptic problem:
\begin{subequations}
\begin{align}
\div(P(\sigma)\nabla \phi^e) &= 0      \quad \text{on } \Omega^e_0, \\
\phi^e &= 0                              \quad \text{on } \Gamma^{\mathrm{ext}}, \\
\tfrac{1}{\det(K)}P(\sigma)\nabla\phi^e\cdot (-\norz) &= \psi   \quad \text{on } \Gamma_0.
\end{align}
\end{subequations}
The Neumann-to-Dirichlet operator $\F^e$ is defined by:
\begin{equation}
\F^e[\rho](\psi) := \phi^e_{|\Gamma_0}\in H^{\tfrac{3}{2}}(\Gamma_0).
\end{equation}
\end{defn}

As for the Dirichlet-to-Neumann operator, $\F^e$ can be extended from functions in $H^{s}$ to function in $H^{s+1}$ (assuming that $\rho$ belongs to $W^{r+1,\infty}(\Gamma_0)$, with $r\geq s+1/2$ and $r\in\mathbb{N}$).
It can also be seen as the inverse of the Dirichlet-to-Neumann operator $\G^e$.
For $\psi \in H^{s+1}(\Gamma_0)$ and $\Psi \in H^{s}(\Gamma_0)$, one has:
\begin{equation}
\G^e[-\rho](\F^e[\rho](\Psi)) = \Psi,
\quad \text{ and } \quad
\F^e[\rho]\left(\G^e[-\rho](\psi)\right) = \psi.
\end{equation}
With Neumann-to-Dirichlet, we can express \eqref{e:psi-phie} as
\begin{equation} \label{e:eq-prot-proof1}
\phi^e_{|\Gamma_0} = \F^e[\rho]\left(|\Nor|\psi\right).
\end{equation}

\begin{rem}
Following the footstep of \cite{Cellule1}, we have the boundary condition
\begin{equation} \label{e:weird-cond}
\phi^e = 0                              \quad \text{on } \Gamma^{\mathrm{ext}},
\end{equation}
and in particular, our Neumann-to-Dirichlet operator is defined on all functions.
Classically for the solution of 
\begin{subequations}
\begin{align}
\Delta \Phi &= 0, \quad \text{on } \omega, \\
\partial_{\nor}\Phi &= \Psi, \quad \text{on } \partial\omega, 
\end{align}
\end{subequations}
one requires the condition 
\begin{equation}
\int_{\partial\omega}{\Psi} = 0,
\end{equation}
in order to satisfy Stoke's theorem and have existence of solution.
Due to \eqref{e:weird-cond}, this is not the case here.
\end{rem}

Let $\alpha_0 = X_0 + \rho_0(X_0)\norz(X_0) \in \Gamma(0)$ and $\alpha(t)$ be the trajectory described by \eqref{e:vitesse-signal}, which can be rewritten as:
\begin{equation} \label{e:vitesse-signal-2}
\tfrac{d\alpha}{dt} =\nabla \Phi^i (\alpha(t)).
\end{equation}
We decompose $\alpha(t)$ under the base $T\Gamma_0 \oplus T^{\perp}\Gamma_0$ of $\tub$:
\begin{equation}
\alpha(t) =: X(t) + \rho(t,X(t)) \norz(X(t)). 
\end{equation}
Then we derive $\alpha$ with respect to time. 
We denote by $\Xp(t) := \tfrac{dX}{dt}\in T_{X(t)}\Gamma_0$.
\begin{equation}
\tfrac{d\alpha}{dt}
    = \Xp 
     + \partial_t\rho \, \norz
     + (\nabla_{\Gamma_0}\rho \cdot \Xp) \, \norz
     + \rho\, d_X\norz (\Xp).
\end{equation}
This can be rewritten, using the matrix $K$ defined in \eqref{e:def-K}, as
\begin{equation}
\tfrac{d\alpha}{dt}
    = \partial_t\rho \, \norz
     + (\nabla_{\Gamma_0}\rho \cdot \Xp) \, \norz
     + K(\Xp).
\end{equation}
Now, we take the scalar product against $\Nor$ defined in \eqref{e:unit-vect}, in order to supress the tangential component:
\begin{equation*}
\nabla \Phi^i \cdot \Nor 
    = \partial_t \rho
     + \nabla_{\Gamma_0}\rho \cdot \Xp
     - K(\Xp) \cdot \,^tK^{-1}(\nabla_{\Gamma_0}\rho). 
\end{equation*}
The last two terms cancel out, we get
\begin{align} \notag
\partial_t\rho 
    &= \nabla^{\sigma} \phi^i \cdot \Nor \\ \label{e:eq-prot-proof2}
    &= \G^i[\rho](\phi^e_{|\Gamma_0}).
\end{align}

Combining \eqref{e:eq-prot-proof1} and \eqref{e:eq-prot-proof2}, we get the following proposition:

\begin{prop} \label{p:eq-prot}
Let $r\in \N$ and  $\rho \in U^{r+1}(\Gamma_0)$.
Let $0\leq s\leq r-1/2$ be a real number.
The operator $\A$ defined by 
\begin{equation}
\A[\rho] : \psi\in H^{s}(\Gamma_0) \mapsto \G^i[\rho](\F^e[\rho](\psi)) \in H^{s}(\Gamma_0),     
\end{equation}
is well defined and bounded.

Moreover, the equation evolution for the boundary of the cell reads as 
\begin{equation}\label{e:prot-equa}
\partial_t\rho - \A[\rho](|\Nor|\psi) = 0.
\end{equation}
\end{prop}

\begin{rem}
Instead of working with the flux of concentration  of enzyme $\Psi$ on $\Gamma$, we are working with the flux $\psi$ over $\Gamma_0$.
Mathematically, this makes sense, as one need to know $\Gamma$, which is our unknown in order to prescribe $\Psi$.
Therefore $\Psi$ is quite an awkward variable to work with, and we prefer the more stable $\psi$.
\end{rem}

\begin{rem}
The matrixes
$$ K^i_{\rho} = \Id + \rho d\norz^i = \Id + (-\rho) d\norz^e = K^e_{-\rho}, $$
are equal, so we write $K$ for one or the other.
\end{rem}

\subsection{Linearization of the equation \eqref{e:prot-equa} around $\rho = 0$.} \label{ss:lin-equa}

Let us give the derivative with respect to the shape of $\A$ and $\F^e$.
At the end, we apply this formula with $\rho = 0$ to get the linearised equation, but we include the general case for the sake of completness.

\begin{prop} \label{p:A-shape-der}
The shape derivative of $\A$ and $\F^e$ with respect to $\rho$ are given by 
\begin{equation} \label{e:der-shape-F}
d_{\rho}\F^e(\rhop)\psi
    := -\rhop w^e 
      - \F^e(\tfrac{1}{\det(K)}\div_{\Gamma_0}(\rhop \Ve))
      + a_1(\rho,\rhop,\psi),
\end{equation}
and 
\begin{equation} \label{e:der-shape-A}
d_{\rho}\A(\rhop)\psi  
    := - \G^i(\rhop \hw)
      - \tfrac{1}{\det(K)}\div(\rhop\tilde{V})
      - (\Id + \A)\left(\tfrac{1}{\det(K)}\div(\rhop \Ve)\right)
      - b_0(\rho,\rhop,\psi),
\end{equation}
where :
\begin{subequations}\label{e:syst-def-term-der-A}
\begin{align} \label{e:der-shape-def-we}
w^e          
    &  
     := \frac{\psi-G\nabla_{\Gamma_0}\rho\cdot\nabla_{\Gamma_0}\F^e\psi}{1+|\,^tK^{-1}\nabla_{\Gamma_0}\rho|^2}, \\ \label{e:der-shape-def-wi}
w^i          
    &  
     := \frac{\A\psi+G\nabla_{\Gamma_0}\rho\cdot\nabla_{\Gamma_0}\F^e\psi}{1+|\,^tK^{-1}\nabla_{\Gamma_0}\rho|^2}, \\ \label{e:der-shape-def-hw}
\hat{w}
    &  
     := \frac{(\Id+\A)\psi}{1+|\,^tK^{-1}\nabla_{\Gamma_0}\rho|^2}, \\ \label{e:der-shape-def-Ve}
\mathbf{V}^e 
    &  
     := \det(K)\, G(\nabla_{\Gamma_0}\F^e\psi+w^e\nabla_{\Gamma_0}\rho), \\ \label{e:der-shape-def-Vi}
\mathbf{V}^i 
    &  
     := \det(K)\, G(\nabla_{\Gamma_0}\F^e\psi-w^i\nabla_{\Gamma_0}\rho), \\ \label{e:der-shape-def-tV}
\mathbf{\tilde{V}}
    &  
     :=\det(K)\, \hat{w}\, G\nabla_{\Gamma_0}\rho, \\ \label{e:der-shape-def-a1}
a_1(\rho, \rhop, \psi)
    &:= \F^e\left( \rhop \left(\frac{ H_0}{1-\rho}+\mathrm{tr}(K^{-1}d\norz)\right)\psi
     - \frac{\rhop}{\det(K)}\mathbf{V}^{e}\cdot d\norz\,^{t}K^{-1}\nabla_{\Gamma_0}\rho\right), \\ \notag
b_0(\rho,\rhop,\psi) 
    &:= +\A \left(\rhop \left(\frac{ H_0}{1-\rho}+\mathrm{tr}(K^{-1}d\norz)\right)\psi\right)
     - \rhop \left(\frac{ H_0}{1+\rho}+\mathrm{tr}(K^{-1}d\norz)\right) \A(\psi) \\ \label{e:der-shape-def-b0}
    &- \A \left(\tfrac{\rhop}{\det(K)}\mathbf{V}^e\cdot d\norz\,^{t}K^{-1}\nabla_{\Gamma_0}\rho\right)
     + \tfrac{\rhop}{\det(K)}\mathbf{V}^i\cdot d\norz\,^{t}K^{-1}\nabla_{\Gamma_0}\rho.
\end{align}
\end{subequations}
\end{prop}

\begin{proof}
This Proposition is an application of Leibniz's rule for operators depending on a parameter.
Let us start from
\begin{equation}
\G^e[-\rho](\F^e[\rho](\psi))= \psi,
\end{equation}
and differentiate it with respect to $\rho$ in the direction $\rhop$.
We get that 
\begin{equation}\label{art4-e:leib-op}
-d_{-\rho}\G^e(\rhop)(\F^e[\rho](\psi)) 
    + \G^e[-\rho](d_{\rho}\F^e(\rhop)(\psi)) 
    = 0,
\end{equation}
Now apply $\F^e$ (which is linear) to the whole equation :
\begin{equation} 
d_{\rho}\F^e(\rhop)(\psi) 
    = \F^e[\rho]\left(d_{-\rho}\G^e(\rhop)\big(\F^e[\rho](\psi)\big)\right).
\end{equation}
Then, we apply Theorem \ref{t:shape-der-formula}, and replace term by term, which gives:
\begin{equation} 
\mathrm{d}_{-\rho}\G^e(\rhop)\Psi 
    = -\G^e[-\rho](\rhop w^e) 
    - \tfrac{1}{\det(K)}\div_{\Gamma_0}(\rhop \mathbf{V}^e)
    + \rhop\, a_0(\rho,\Psi),
\end{equation}
with
\begin{subequations}
\begin{align}
w^e          
    & = \frac{\G^e[-\rho]\Psi-G\nabla_{\Gamma_0}\rho\cdot\nabla_{\Gamma_0}\Psi}{1+|\,^tK^{-1}\nabla_{\Gamma_0}\rho|^2}, \\
\mathbf{V}^e 
    &  = \det(K)\, G(\nabla_{\Gamma_0}\Psi+w^e\nabla_{\Gamma_0}\rho), \\
a_0(\rho, \psi)
    &:= +\left(\frac{ H_0}{1-\rho}+\mathrm{tr}(K^{-1}d\norz)\right)\G^{e}[-\rho]\Psi
     - \frac{1}{\det(K)}\mathbf{V}^{e}\cdot d\norz\,^{t}K^{-1}\nabla_{\Gamma_0}\rho,
\end{align}
\end{subequations}
Then we plug in $\F^e[\rho](\psi)$ instead of $\psi$ and apply $\F^e$ and get \eqref{e:der-shape-F} with $w^e$, $\Ve$ and $a_1$ defined through \eqref{e:der-shape-def-we}, \eqref{e:der-shape-def-Ve} and \eqref{e:der-shape-def-a1}.

Similarly, by Leibniz's rule, we get 
\begin{equation}
d_{\rho}\A(\rhop)(\Psi) 
    = d_{\rho}\G^i(\rhop)\big(\F^e[\rho](\psi)\big)
     + \G^i[\rho]\big(d_{\rho}\F^e(\rhop)(\psi)\big).
\end{equation}
Which when we write it term by term using the definitions of \eqref{e:syst-def-term-der-A} gives:
\begin{align*}
d_{\rho}\A(\rhop)(\Psi) 
    =& - \G^i(\rhop w^i) 
       -\tfrac{1}{\det(K)}\div(\rhop \Vi) 
       + \rhop a_0(\rho,\F^e\psi) \\
     & - \G^i(\rhop w^e)
       - \G^i\F^e\left(\tfrac{1}{\det(k)}\div_{\Gamma_0}(\rhop \Ve)\right)
       + \G^i(a_1(\rho,\rhop,\psi)) \\
    =& - \G^i(\rhop (w^i+w^e))
       - \tfrac{1}{\det(K)}\div(\rhop (\Vi-\Ve)) 
       - \tfrac{1}{\det(K)}\div(\rhop \Ve) \\
     & - \A\left(\tfrac{1}{\det(k)}\div_{\Gamma_0}(\rhop \Ve)\right)
       + \rhop a_0(\rho,\F^e\psi) 
       + \G^i(a_1(\rho,\rhop,\psi)) \\
    =& - \G^i(\rhop \hw)
       - \tfrac{1}{\det(K)}\div(\rhop \tiV)
       - (\Id+\A)\left(\tfrac{1}{\det(k)}\div_{\Gamma_0}(\rhop \Ve)\right)
       + b_0(\rho,\rhop,\psi).
\end{align*}
\end{proof}

\begin{rem}
In \eqref{art4-e:leib-op}, we differentiate like a product and note like a composition, because the derivative is not on the variable over which we are doing a composition, but over $\rho$.
Thus, we need to differentiate something that is mainly a tensorial product.
\end{rem}

The first step in solving an equation of the form of \eqref{e:prot-equa} is to look at the linearised equation associated to it.
This means writing the parametrisation of the boundary $\rho$ under the form $\rho = \rho_0 +\varepsilon\rhop$ and look at the equation verified by $\rhop$ when $\varepsilon$ goes to $0$.
By definition of the shape derivative, this equation can be rewritten as follows
\begin{equation}
\partial_t \rhop - d_{\rho_0}\A(\rhop)(|\Nor|\psi) = \A[\rho_0](d_{\rho_0}(|N|)(\rhop) \, \psi).
\end{equation}
Without loss of generality, since the reference domain $\Gamma_0$ can be anything, we look at the case $\rho_0 = 0$.
We have the following Corollary of Proposition \ref{p:A-shape-der}:

\begin{cor} \label{c:prot-linearised}
The linearization of \eqref{e:prot-equa} around $\rho_0 =0$ can be written as 
\begin{equation} \label{e:prot-linearised}
\partial_t\rhop 
    + \A_1(\rhop,\psi)
    = 0.
\end{equation}
Here the operator $\A_1 = d_0 \A$ is given by
\begin{equation}
\A_1(\rhop,\psi) 
    = \G^i_0(\rhop\hw)
    + (\Id+\A_0)\div(\rhop \mathbf{V}^e)
    + b_0(\rhop,\psi),
\end{equation}
where $\G^i_0 := \G^i[0]$, $\F^e_0:= \F^e[0]$, $\A_0 := \A[0]$ and  $\tw$, $\mathbf{V}_e$ and $b_0$ are given by \eqref{e:syst-def-term-der-A}, namely:
\begin{subequations}
\begin{align}
\tilde{w}
    &  
     = (\Id+\A_0)\psi, \\
\mathbf{V}^e 
    &  
     = \nabla_{\Gamma_0}\F^e_0\psi, \\
b_0(\rhop,\psi)
    &= \A_0(2H_0\rhop \psi) - 2H_0\rhop \A_0(\psi).
\end{align}
\end{subequations}
\end{cor}

\begin{proof}
Since $|\Nor| = \sqrt{1+\left|\,^tK^{-1}\nabla_{\Gamma_0}\rho\right|^2}$, we have 
$$d_{\rho}(|\Nor|)(\rhop) = \tfrac{\,^tK^{-1}\nabla\rhop\cdot\,^tK^{-1}\nabla\rho}{\sqrt{1+\left|\,^tK^{-1}\nabla_{\Gamma_0}\rho\right|^2}}$$ 
and in particular for $\rho=0$, we get $d_{0}(|N|)(\rhop) = 0$ for all $\rhop$.
\end{proof}

\begin{rem}
The term $(H_0-(1+\rho)\mathrm{tr}(K^{-1}d\norz))$, which in a way was the difference between the curvature $\Gamma_0$ and $\Gamma$, is equal to zero when $\rho=0$.
\end{rem}


\section{Well posedness of the linearised equation in $H^1$} \label{s:wp}


\subsection{Regularising properties for $\Id-\A_0$} \label{ss:reg-prop}

When we look at $\A_1$, we have two terms that we need to control in order to study \eqref{e:prot-linearised}.
The first one is $\G^i_0(\rhop \hw)$: this term is parabolic, and therefore mostly useful and not detrimental in the energy estimates.
The second one is  $(\Id+ \A_0)\div(\rhop \Ve)$.
Normally divergence terms in energy estimates are treated through skew-symmetry properties.
Due to the anti-symmetric nature of $\div$, when doing energy estimates, those terms do vanish, up to lower order terms.
The presence of $\Id + \A_0$ make this harder to prove, as one has to bound the commutator between $\A_0$ and $\div$ to use the anti-symmetry.
Instead of doing that, we will prove here that $\Id-\A_0$ is a regularising operator, allowing us to cancel the part of this term that is less clearly skew-adjoint.
Although this technique is powerful, as it allows us to skip the commutation estimates, it is less easy to generalise.
When one tries to study the full equation \eqref{e:prot-equa} through quasi-linearisation, one need good estimates in $\rho$ for the solution of \eqref{e:prot-linearised}, estimates which seem harder to obtain in this way. \\

The following lemma gives a first order approximation of the Dirichlet-to-Neumann operator.
It is a direct consequence of \cite{TaylorII}, Chapter VII, paragraph 11 and 12.

\begin{lem}\label{l:DtN-approx}
Let $\omega$ be a bounded domain with a boundary $\partial\omega$ which is $C^{\infty}$.
We assume that the boundary consists ok two disjoin connected components $\Gamma_{\mathrm{int}}$ and $\Gamma_{\mathrm{ext}}$ (with potentially $\Gamma_{\mathrm{ext}} = \emptyset$).
Let $\G$ be the Dirichlet to Neumann operator on $\Gamma_{\mathrm{int}}$ with homogeneous boundary condition on $\Gamma_{\mathrm{ext}}$.
Let $\Delta_{\Gamma_{\mathrm{int}}}$ be the Laplace-Beltrami operator on $\Gamma_{\mathrm{int}}$. 
Then there exists an operator $R$ such that :
\begin{equation}
\G = \sqrt{-\Delta_{\Gamma_{\mathrm{int}}}} + R,
\end{equation}
which is of order $0$, in the sense that $R : H^s(\Gamma_{\mathrm{int}}) \rightarrow H^s(\Gamma_{\mathrm{int}})$ is bounded for all $s\geq 0$.
\end{lem}

Now using the fact that on both domains $\Omega^i$ and $\Omega^e$ the active part of the boundary for the Dirichlet to Neumann operator is the same, namely $\Gamma_0$,
and the fact that the Laplace-Beltrami operator of the boundary does not depend on the orientation of the interface, we have the following corollary:

\begin{cor} \label{c:reg-GepGi}
The operator $\mathcal{R} := \G^i_0 - \G^e_0$ is of order $0$, in the sense that $\mathcal{R}_0 : H^s(\Gamma_{0}) \rightarrow H^s(\Gamma_{0})$ is bounded for all $s\geq 0$.
\end{cor}

In particular, when composing with $\F^e_0\circ\div$, which is also an operator of order $0$, we~obtain

\begin{cor} \label{c:reg-Id-A_0-div}
The operator $\mathcal{R}_1:=(\Id-\A_0) \div_{\Gamma_0}$ is of order $0$, in the sense that $\mathcal{R}_1 : H^s(\Gamma_{0}) \rightarrow H^s(\Gamma_{0})$ is bounded for all $s\geq 0$.
\end{cor}

In particular, this corollary means that we can rewrite $\A_1$ as follows:
\begin{equation} \label{e:A1-decomp}
\A_1(\rhop,\psi) 
    = \G^i_0(\rhop\hw)
    + 2\, \div(\rhop \mathbf{V}^e)
    + \tilde{b_0}(\rhop,\psi),
\end{equation}
where $\Vert \tilde{b_0}(\rhop,\psi) \Vert_{H^s} \leq C_s \Vert\rhop\Vert_{H^s}$.

\subsection{Gärding type $L^2$ estimates} \label{ss:L2-est}

In this paragraph, we provide an $L^2$-estimate associated with \eqref{e:prot-linearised}.
This estimate will be used to recover the $H^1$-uniqueness later on.

\begin{prop} \label{p:energy-estimates}
Let $\psi\in  L^{\infty}([0,T], W^{2,\infty}(\Gamma_0))$ be such that the function $\hw := (\Id+\A_0)\psi$ is bounded from below by a  positive constant $\hw\geq \alpha>0$. 
For $\rhop\in H^{1}(\Gamma_0)$,
we have the inequality
\begin{equation}
\langle \A_1(\rhop,\psi), \rhop \rangle
    \geq c\alpha \Vert \rhop\Vert_{H^{1/2}}^2 
     - C\Vert\rhop\Vert_{L^2}^2, 
\end{equation}
for some constant $C>0$ depending on $\Vert \psi\Vert_{W^{2,\infty}}$, $\alpha$ and $\Omega_0$ and some constant $c>0$ depending on $\Omega_0$.
\end{prop}
\begin{proof}[Proof of Proposition \ref{p:energy-estimates}]
We decompose $\A_1$ following \eqref{e:A1-decomp}, we have three terms to control here:
\begin{equation} \label{e:L2-est-test-funct}
\langle \A_1(\rhop,\psi),  \rhop \rangle
    = \langle \G^i_0(\rhop\hw),\rhop \rangle
    + 2\langle \div(\rhop \Ve), \rhop \rangle
    + \langle \tilde{b_0}(\rhop,\psi), \rhop \rangle.
\end{equation}

\begin{itemize}
\item 
For the first term of \eqref{e:L2-est-test-funct}, we start by using the symmetry of the Dirichlet-to-Neumann operator:
\begin{equation}
\langle \G^i_0(\rhop\hw),\rhop \rangle
    = \langle \rhop\hw,\G^i_0(\rhop) \rangle.
\end{equation}
We fix an extension $e(\hw)$ of $\hw$ to $W^{2,\infty}(\Omega_0)$ that also respect $e(\hw)\geq \alpha$.
To construct $e(\hw)$, we introduce a $C^{\infty}$ cutoff function $\delta$, with $0\leq \delta \leq 1$,  $\delta_{|\Gamma_0} = 1$, and $\delta_{|\Omega_0\setminus\tub} = 0$ and set:
\begin{equation}
e(\hw)(X+x\norz(X)) 
    = \delta(X+x\norz(X)) (\hw(X) - \alpha) 
     + \alpha.
\end{equation}
Then, one can see that $e(\hw)\phi_0(\rhop)$ is an extension of $\hw\rhop$ and therefore, by stokes theorem, we have:
\begin{align} \notag
\langle \rhop\hw,\G^i_0(\rhop) \rangle
    &= \int_{\Omega_0}{\nabla(e(\hw)\phi_0(\rhop)) \cdot \nabla\phi_0(\rhop)} \\
    &= \int_{\Omega_0}{\phi_0(\rhop) \nabla(e(\hw)) \cdot \nabla\phi_0(\rhop)} 
     + \int_{\Omega_0}{e(\hw) |\nabla\phi_0(\rhop)|^2}.
\end{align}
Then, we bound the first term using $2ab \leq \varepsilon a^2 + \tfrac{1}{\varepsilon} b^2$ by
\begin{equation*}
\left|\int_{\Omega_0}{\phi_0(\rhop) \nabla(e(\hw)) \cdot \nabla\phi_0(\rhop)} \right|
    \leq \tfrac{2}{\alpha}\Vert e(\hw)\Vert_{W^{1,\infty}} \Vert \phi_0(\rhop)\Vert_{L^2}^2 
     + \tfrac{\alpha}{2} \Vert \nabla\phi_0(\rhop)\Vert_{L^2}^2.
\end{equation*}
And we bound the second by below using the hypothesis on $e(\hw)$:
\begin{equation}
\int_{\Omega_0}{e(\hw) |\nabla\phi_0(\rhop)|^2} 
    \geq \alpha \Vert \nabla\phi_0(\rhop)\Vert_{L^2}^2.
\end{equation}
We recall that the operator $\phi_0 : L^2(\Gamma_0) \rightarrow H^{1/2}(\Omega_0)$ is continuous (see \cite{TaylorI}, Chapter $V$, Proposition $1.8$ for example) and in particular 
\begin{equation}
\Vert \phi_0(\rhop)\Vert_{L^2} \leq \Vert \rhop\Vert_{L^2}.
\end{equation}
We also recall that due to trace theorem, the $H^1$ norm of $\phi_0(\rhop)$ control the $H^{1/2}$ norm of its trace
\begin{equation}
\Vert \nabla\phi_0(\rhop)\Vert_{L^2}^2
    = \Vert \phi_0(\rhop)\Vert_{H^1}^2 
     - \Vert \phi_0(\rhop)\Vert_{L^2}^2
    \geq c\Vert \rhop\Vert_{H^{1/2}}^2
     - \Vert\rhop \Vert_{L^2}^2
\end{equation}
Therefore :
\begin{equation} \label{e:L2-est-coerc-term}
\langle \rhop\hw,\G^i_0(\rhop) \rangle
    \geq c\alpha \Vert\rhop\Vert_{H^{1/2}}^2
     - C\Vert \rhop\Vert_{L^2}^2
\end{equation}

\item 
For the second term of \eqref{e:L2-est-test-funct} we use the anti-symmetry to bound it:
\begin{align*}
\int_{\Gamma_0}{\rhop\, \div(\rhop\Ve)}
   &= - \int_{\Gamma_0}{\rhop\Ve \cdot \nabla \rhop} \\
   &= - \int_{\Gamma_0}{\rhop \left(\div(\rhop\Ve)-\rhop\, \div(\Ve)\right)},
\end{align*}
and therefore :
\begin{equation}
\int_{\Gamma_0}{\rhop\, \div(\rhop\Ve)}
    = \tfrac{1}{2} \int_{\Gamma_0}{\rhop^2\div(\Ve)}.
\end{equation}
To conclude, we have 
\begin{equation} \label{e:L2-est-anti-sym}
\left|\int_{\Gamma_0}{\rhop\, \div(\rhop\Ve)}\right|
    \leq \Vert \Ve\Vert_{W^{1,\infty}} \Vert\rhop\Vert_{L^2}^2.
\end{equation}
\item 
The last term is of order $0$:
\begin{equation} \label{e:L2-est-last-term}
\left|  \langle \tilde{b_0}(\rhop,\psi), \rhop \rangle \right|
    \leq C \Vert\rhop\Vert_{L^2}^2.
\end{equation}
\end{itemize}
By using \eqref{e:L2-est-coerc-term}, \eqref{e:L2-est-anti-sym}, \eqref{e:L2-est-last-term} into \eqref{e:L2-est-test-funct}, we get 
\begin{equation}
\langle \A_1(\rhop,\psi), \rhop \rangle
    \geq c\alpha \Vert \rhop\Vert_{H^{1/2}}^2 
     - C\Vert\rhop\Vert_{L^2}^2, 
\end{equation}
as wanted.
\end{proof}

\subsection{Gärding type $\dot{H}^1$ estimates} \label{ss:H1-est}

In this paragraph, we provide an $H^1$-estimate associated with \eqref{e:prot-linearised}.
This estimate will be used to prove the existence of $H^1$-solution to \eqref{e:prot-linearised}.
The standard approach to obtain $H^1$ estimates is to start by integrating against the laplacian and try to rearrange the derivative nicely.
As seen in the case of the $L^2$-estimate, an easy way to control the term involving $G^i$ is to go from $\Gamma_0$ to $\Omega_0$ using Stokes formula and make the manipulation in the full domain instead of the boundary.
We wish to do the same here.
To that end, we construct an extension of the boundary laplacian in the tubular neighbourhood $\tubp$.

Without loss of generality (we can restrict ourself to an even smaller neighbourhood), we assume that $\tubp$ is of the form:
\begin{equation}
\tubp = \{ X+x\norz(X); X\in\Gamma_0, x\in[-\varepsilon_1,\varepsilon_1]\}.
\end{equation}
We fix $\delta_1 : [-\varepsilon_1,0] \rightarrow [0,1]$ a $C^{\infty}$ cutoff function that is equal to zero in a neighborhood of $-\varepsilon_1$ and $1$ in a neighbourhood of $0$.
We denote by $\Sigma$ the diffeomorphism that links $\tubp$ with its normal representation:
\begin{equation}
\Sigma : (X,x) \in \Gamma_0\times [-\varepsilon_1,\varepsilon_1] \mapsto X+x\norz(X) \in \tubp,
\end{equation}
and by $J_{X,x}$ its jacobian determinant:
\begin{equation}
J_{X,x} := \det(\Id+xL_X).
\end{equation}
Armed with this notation, we define the extensions of the tangential gradient and divergence operator as follows:

\begin{defn}
Let $f \in C^1(\tub)$ and $v\in C^1(\tub,\R^{d+1})$. 
We define $\nabla_{\Gamma_0}^{\dagger}f$ and $\div_{\Gamma_0}^{\dagger}$ as 
\begin{equation}
\nabla_{\Gamma_0}^{\dagger}f(X+x\norz(X)) 
    := \nabla_{\Gamma_0} f_x(X),
\end{equation}
and 
\begin{equation}
\div_{\Gamma_0}^{\dagger}(v)(X+x\norz(X))
   := J_{X,x} \div_{\Gamma_0}(J_{\cdot,x}^{-1} v_x) (X),
\end{equation}
where 
\begin{equation*}
f_x(X) := f(X+ x\norz(X))
\quad \text{ and } \quad
v_x(X) := \Pi_0(v(X+x\norz(X))),
\end{equation*}
where $\Pi_0$ is the projection on $T_X\Gamma_0$ defined in \eqref{art4-e:Pi0_def}.

We denote by $\Delta_{\Gamma_0}^{\dagger} := \div_{\Gamma_0}^{\dagger} \circ \nabla_{\Gamma_0}^{\dagger}$ the extension of the boundary laplacian corresponding to the one we took for the gradient and divergence operators.
\end{defn}

\begin{rem}
For $f\in C^1(\Omega_0,\R)$, and $\chi \in C^{\infty}(\Omega_0,\R)$ equal to $0$ outside of $\tub$, we can see $\chi f$ as a function over $\tub$ and therefore define its extended gradient $\nabla_{\Gamma_0}^{\dagger}(\chi f)$.
Similarly, we can define the divergence of $\chi \mathbf{v}$ for $\mathbf{v}$ any vector field over $\Omega_0$. 
\end{rem}

One of the usefull properties of the extension we took for the gradient and divergence is that they are adjoints one of the other.

\begin{lem} \label{l:ipp-dagger}
Let $f \in H^1(\tub)$ and $v\in H^1(\tub,\R^{d+1})$. 
One has 
\begin{equation}
\int_{\Omega_0}{f \, \div_{\Gamma_0}^{\dagger}(v) }
    = - \int_{\Omega_0}{\nabla_{\Gamma_0}^{\dagger}f \cdot v}.
\end{equation}
\end{lem}

\begin{proof}
We use the change of variable $\Sigma$, as well as Stokes formula on $\Gamma_0$:
\begin{align*}
\int_{\Omega_0}{f \, \div_{\Gamma_0}^{\dagger}(v) }
    &= \int_{\tubp\cap \Omega_0}{f \, \div_{\Gamma_0}^{\dagger}(v) } \\
    &= \int_{-\varepsilon_1}^{0}{\int_{\Gamma_0}{\left(f \, \div_{\Gamma_0}^{\dagger}(v)\right)(X+x\norz(X)) J_{X,x} \, \mathrm{d}X} \, \mathrm{d}x} \\
    &= \int_{-\varepsilon_1}^{0}{\int_{\Gamma_0}{f_x(X) \, \div_{\Gamma_0}( J_{\cdot,x}\Pi_0 v_x)(X) \, \mathrm{d}X} \, \mathrm{d}x} \\
    &= \int_{-\varepsilon_1}^{0}{\int_{\Gamma_0}{f_x(X) \, \div_{\Gamma_0}( J_{\cdot,x}\Pi_0 v_x)(X) \, \mathrm{d}X} \, \mathrm{d}x} \\
    &= -\int_{-\varepsilon_1}^{0}{\int_{\Gamma_0}{\nabla_{\Gamma_0}f_x (X) \cdot \Pi_0 v_x(X)  J_{X,x} \,\mathrm{d}X} \, \mathrm{d}x} \\
    &= -\int_{-\varepsilon_1}^{0}{\int_{\Gamma_0}{\left(\nabla_{\Gamma_0}^{\dagger}f \cdot v \right)(X+x\norz(X))  J_{X,x} \, \mathrm{d}X} \, \mathrm{d}x} \\
    &= -\int_{\tubp\cap \Omega_0}{\nabla_{\Gamma_0}^{\dagger}f \cdot v } \\
    &= -\int_{\Omega_0}{\nabla_{\Gamma_0}^{\dagger}f \cdot v }.
\end{align*}
\end{proof}

One other useful property of the tangential gradient is that $\nabla_{\Gamma_0}^{\dagger}\nabla \phi_0(\rhop)$ bounds $\rhop$ nicely.

\begin{lem}\label{l:coerc-nab-dagger-nab}
Let $\chi\in C^{\infty}(\Omega_0)$ be a cutoff function equal to $0$ outside $\tub$ and equal to $1$ on $\tubp$.
Let $\rhop \in H^{3/2}(\Gamma_0)$.
There exists a constant $c>0$, independent of $\rhop$ such that
\begin{equation}
\Vert\nabla_{\Gamma_0}^{\dagger}\nabla \left(\chi\phi_0(\rhop)\right) \Vert_{L^2(\Omega_0)}^2
    + \Vert\rhop\Vert_{L^2(\Gamma_0)}^2 
    \geq c\Vert \rhop\Vert_{H^{3/2}(\Gamma_0)}^2.
\end{equation}
\end{lem}

\begin{proof}
The proof uses three ideas.
The first one is that, since $\phi_0(\rhop)$ is harmonic, it is smooth on any compact included in $\Omega_0$ (see \cite{Mclean}, Theorem 4.16 for example).
In particular, there exists a constant $C$, independent of $\rhop$, such that:
\begin{equation}
\Vert \chi\phi_0(\rhop)\Vert_{H^2(\tub\setminus \tubp)}
    \leq C\Vert\rhop\Vert_{L^2(\Gamma_0)}.
\end{equation}
The second idea is that since $\phi_0(\rhop)$ is harmonic, the normal derivative plus the tangential ones gives $0$.
More precisely, on $\tubp$, $\chi=1$, and one has:
\begin{equation}
\div_{\Gamma_0}(\nabla_{\Gamma_0} \phi_0(\rhop))
     + (H_0+\partial_{\norz}) \partial_{\norz} \phi_0(\rhop)
    = 0,
\end{equation}
which gives
\begin{equation}
\partial_{\norz}^2 \phi_0
    = - H_0 \partial_{\norz}\phi_0(\rhop) 
      - \Delta_{\Gamma_0}^{\dagger} \phi_0(\rhop)
      + J_{X,x} \nabla_{\Gamma_0} J_{\cdot,x} \cdot \nabla_{\Gamma_0}^{\dagger} \phi_0(\rhop).
\end{equation}
In particular, since 
$$\Vert \nabla^2\phi_0(\rhop) \Vert_{L^2(\tubp)} 
    \leq \Vert \nabla_{\Gamma_0}^{\dagger}\nabla\phi_0(\rhop)\Vert_{L^2(\tubp)} 
     + \Vert \partial_{\norz}^2\phi_0(\rhop)\Vert_{L^2(\tubp)},$$
we get that for $C \geq \Vert H_0\Vert_{L^{\infty}}$:
\begin{equation}
\Vert \nabla^2\phi_0(\rhop) \Vert_{L^2(\tubp)} 
    \leq 2\Vert \nabla_{\Gamma_0}^{\perp}\nabla\phi_0(\rhop)\Vert_{L^2(\tubp)} 
     + C\Vert \phi_0(\rhop)\Vert_{H^1(\tubp)}.
\end{equation}
And the third idea is to conclude using the trace theorem:
\begin{equation}
\Vert\rhop\Vert_{H^{3/2}} 
    \leq C\Vert\phi_0(\rhop)\Vert_{H^2}.
\end{equation}
\end{proof}

Armed with this two Lemmata, we can now work on the $H^1$ estimates.

\begin{prop} \label{p:H1-estimates}
Let $\psi\in  W^{3,\infty}([0,T]\times \Gamma_0))$ be such that the function $\hw := (\Id+\A_0)\psi$ is bounded by below by a  positive constant $\hw\geq \alpha>0$. 
Let $\rhop\in C^{\infty}(\Gamma_0)$.
Then we have the inequality
\begin{equation}
\langle \A_1(\rhop,\psi), -\Delta_{\Gamma_0} \rhop \rangle
    \geq c\alpha \Vert \rhop\Vert_{\dot{H}^{3/2}}^2 - C\Vert\rhop\Vert_{H^1}^2, 
\end{equation}
for some constant $C>0$ depending on $\Vert \psi\Vert_{W^{3,\infty}}$ and $\Omega_0$ and some constant $c>0$ depending on $\Omega_0$.
\end{prop}
\begin{proof}[Proof of Proposition \ref{p:H1-estimates}]
Similarly to the case of $L^2$-estimate, we have three terms to control here:
\begin{equation} \label{e:H1-est-test-funct}
\langle \A_1(\rhop,\psi), -\Delta_{\Gamma_0} \rhop \rangle
    = \langle \G^i_0(\rhop\hw), -\Delta_{\Gamma_0}\rhop \rangle
    + 2\langle \div(\rhop \Ve), -\Delta_{\Gamma_0}\rhop \rangle
    + \langle \tilde{b_0}(\rhop,\psi), -\Delta_{\Gamma_0}\rhop \rangle.
\end{equation}
\begin{itemize}
\item For the first term of \eqref{e:H1-est-test-funct}, the first step is once again to go onto the full domain.
Let us fix a cutoff function $\chi \in C^{\infty}(\Omega_0)$ which is equal to zero outside of $\tub$ and equal to $1$ on $\tubp$.
We extend $\Delta_{\Gamma_0}\rhop$ onto the full domain through $\Delta_{\Gamma_0}^{\dagger}(\chi^2\phi_0(\rhop))$ and get
\begin{align} \notag
\langle \G^i(\rhop\hw), \Delta_{\Gamma_0}\rhop \rangle
    &= \int_{\Gamma_0}{\Delta_{\Gamma_0}\rhop \, \G^i_0(\rhop\hw)} \\ \notag
    &= \int_{\Omega_0}{\nabla\phi_0(\rhop\hw)\cdot \nabla \Delta_{\Gamma_0}^{\dagger}(\chi^2\phi_0(\rhop))} \\ \label{e:H1-proof-eq1}
    &= \int_{\Omega_0}{\nabla\phi_0(\rhop\hw)\cdot  \Delta_{\Gamma_0}^{\dagger} \nabla(\chi^2\phi_0(\rhop))}
     + \int_{\Omega_0}{\nabla\phi_0(\rhop\hw)\cdot [\nabla, \Delta_{\Gamma_0}^{\dagger}] (\chi^2\phi_0(\rhop))}
\end{align}
One can immediately remark, 
that 
\begin{equation}
\Vert [\nabla, \Delta_{\Gamma_0}^{\dagger}] (\chi^2\phi_0(\rhop))\Vert_{L^2} 
    \leq C\Vert\nabla_{\Gamma_0}^{\perp}\nabla (\chi\phi_0(\rhop)) \Vert_{L^2},
\end{equation}
for some constant $C$ independent of $\phi_0(\rhop)$.
Using this, the last term of \eqref{e:H1-proof-eq1} can be bounded as follows
\begin{equation} \label{e:H1-proof-coerc-term1-bound}
\left|\int_{\Omega_0}{\nabla\phi_0(\rhop\hw)\cdot [\nabla, \Delta_{\Gamma_0}^{\dagger}] (\chi^2\phi_0(\rhop))} \right|
    \leq \tfrac{4C}{\alpha}\Vert\rhop\Vert_{H^{1/2}}^2
     + \tfrac{\alpha}{4} \Vert\nabla_{\Gamma_0}^{\perp}\nabla (\chi\phi_0(\rhop)) \Vert_{L^2}^2.
\end{equation}

Then, we integrate by part, using Lemma \ref{l:ipp-dagger}, to distribute the derivatives:
\begin{equation}
\int_{\Omega_0}{\nabla\phi_0(\rhop\hw)\cdot  \Delta_{\Gamma_0}^{\dagger} \nabla(\chi^2\phi_0(\rhop))}
    = -\int_{\Omega_0}{\nabla_{\Gamma_0}^{\dagger}\nabla\phi_0(\rhop\hw)\cdot  \nabla_{\Gamma_0}^{\dagger} (\chi^2\nabla\phi_0(\rhop))}.
\end{equation}
We use once again the extension $e(\hw)$ of $\hw$ that we introduced in the proof of Lemma \ref{p:energy-estimates}.
We write $\nabla\phi_0(\rhop\hw)$ as a sum of commutators plus $e(\hw)\nabla\phi_0(\rhop)$, term that will be easier to bound by below later on.
This time, since $\psi$ is more regular, $e(\hw)$ lies in $W^{3,\infty}(\Omega_0)$, we recall that it also verifies $e(\hw)\geq \alpha$.
Then we write:
\begin{equation} \label{e:H1-est-commut-decomp}
\nabla\phi(\rhop\hw) 
   = e(\hw)\nabla_{\Gamma_0}^{\dagger}\nabla\phi_0(\rhop) 
    + [\nabla_{\Gamma_0}^{\dagger}\nabla,e(\hw)]\phi_0(\rhop)
    + \nabla_{\Gamma_0}^{\dagger}\nabla\tilde{\phi}(\rhop,\hw),
\end{equation}
where $\tilde{\phi}$ is defined through:
\begin{equation}
\tilde{\phi}(\rhop,\hw) 
    := \phi_0(\rhop\hw)
     - e(\hw)\phi_0(\rhop). 
\end{equation}
We have now three terms to control
\begin{itemize}
\item The first one, involving $e(\hw)\nabla_{\Gamma_0}^{\dagger}\nabla\phi_0(\rhop)$ is actually positive, up to a commutator 
\begin{equation*}
\int_{\Omega_0}{e(\hw)\nabla_{\Gamma_0}^{\dagger}\nabla\phi_0(\rhop) \cdot \nabla_{\Gamma_0}^{\dagger}\nabla(\chi^2\phi_0(\rhop))}
    = \int_{\Omega_0}{e(\hw)\nabla_{\Gamma_0}^{\dagger}\nabla(\chi\phi_0(\rhop)) \cdot \nabla_{\Gamma_0}^{\dagger}\nabla(\chi\phi_0(\rhop))}
    + R(\chi,\phi_0(\rho),e(\hw)),
\end{equation*}
with 
\begin{equation}
|R(\chi,\phi_0(\rho),e(\hw))|
    \leq C \Vert e(\hw)\Vert_{L^{\infty}} \Vert \phi_0(\rhop)\Vert_{H^1} \Vert \nabla_{\Gamma_0}^{\dagger}\nabla(\chi \phi_0(\rhop))\Vert_{L^2},
\end{equation}
for some constant $C>0$ independant of $\phi_0(\rhop)$ and $e(\hw)$.
We bound it by below, using Lemma \ref{l:coerc-nab-dagger-nab}, as well as the assumption we made on $e(\hw)$:
\begin{align} \notag
\int_{\Omega_0}{e(\hw)\nabla_{\Gamma_0}^{\dagger}\nabla(\chi\phi_0(\rhop)) \cdot \nabla_{\Gamma_0}^{\dagger}\nabla(\chi\phi_0(\rhop))}
    &= \int_{\Omega_0}{e(\hw)\left|\nabla_{\Gamma_0}^{\dagger}\nabla\phi_0(\rhop)\right|^2} \\ \notag
    &\geq  \alpha \int_{\Omega_0}{\left|\nabla_{\Gamma_0}^{\dagger}\nabla\phi_0(\rhop)\right|^2} \\ \label{e:H1-proof-coerc-term2-bound}
    &\geq c\alpha \Vert\rhop\Vert_{H^{3/2}}^2 - C \Vert\rhop\Vert_{L^2}^2
\end{align}
\item The second one, involving 
$[\nabla_{\Gamma_0}^{\dagger}\nabla,e(\hw)]\phi_0(\rhop)$, which is a commutator, can be bounded accordingly.
One has 
\begin{equation}
\Vert [\nabla_{\Gamma_0}^{\dagger}\nabla,e(\hw)]\phi_0(\rhop)\Vert_{L^2}
    \leq \Vert \phi_0(\rhop)\Vert_{H^1},
\end{equation}
and therefore
\begin{equation} \label{e:H1-proof-coerc-term3-bound}
\left|\int_{\Omega_0}{[\nabla_{\Gamma_0}^{\dagger}\nabla,e(\hw)]\phi_0(\rhop)\cdot \nabla_{\Gamma_0}^{\dagger}\nabla\phi_0(\rhop)} \right|
    \leq \tfrac{4C}{\alpha}\Vert\rhop\Vert_{H^{1/2}}^2
     + \tfrac{\alpha}{4} \Vert\nabla_{\Gamma_0}^{\perp}\nabla (\chi^2\phi_0(\rhop)) \Vert_{L^2}^2.
\end{equation}
\item The third one, involving $\nabla_{\Gamma_0}^{\dagger}\nabla\tilde{\phi}(\rhop,\hw)$ is also a commutator term in a way.
To control it we notice that $\tilde{\phi}$ is solution of the following equation: 
\begin{align*}
\Delta\tilde{\phi}(\rhop,\hw)
    &= -\Delta (e(\hw)\phi_0(\rhop)) \\
    &= -2\nabla\phi_0(\rhop)\cdot\nabla e(\hw)
     - \phi_0(\rhop) \Delta e(\hw).
\end{align*}
Moreover $\tilde{\phi}(\rhop,\hw)_{|\Gamma_0}=0$, 
hence by elliptic regularity, one gets:
\begin{equation}
\Vert \tilde{\phi}(\rhop,\hw)\Vert_{H^2}
    \leq C_0\Vert e(\hw)\Vert_{W^{2,\infty}} \Vert \phi_0(\rhop)\Vert_{H^1},
\end{equation}
and therefore
\begin{equation} \label{e:H1-proof-coerc-term4-bound}
\left|\int_{\Omega_0}{\nabla_{\Gamma_0}^{\dagger}\nabla \tilde{\phi}(\rhop,\hw)\cdot \nabla_{\Gamma_0}^{\dagger}\nabla\phi_0(\rhop)} \right|
    \leq \tfrac{4C}{\alpha}\Vert\rhop\Vert_{H^{1/2}}^2
     + \tfrac{\alpha}{4} \Vert\nabla_{\Gamma_0}^{\perp}\nabla (\chi^2\phi_0(\rhop)) \Vert_{L^2}^2.
\end{equation}
\end{itemize}

By combining \eqref{e:H1-proof-coerc-term1-bound}, \eqref{e:H1-proof-coerc-term2-bound}, \eqref{e:H1-proof-coerc-term3-bound} and \eqref{e:H1-proof-coerc-term4-bound}, we get 
\begin{equation} \label{e:H1-est-coerc-term}
\langle \G^i_0(\rhop \hw), -\Delta_{\Gamma_0}\rhop\rangle
    \geq \tfrac{\alpha}{4} \Vert\nabla_{\Gamma_0}^{\perp}\nabla \phi_0(\rhop) \Vert_{L^2}^2
     - C\Vert \rhop\Vert_{H^1}^2.
\end{equation}

\item For the second term of \eqref{e:H1-est-test-funct}, we use the fact that 
\begin{equation*}
\nabla_{\Gamma_0} \left(\frac{1}{2}|\nabla_{\Gamma_0} \rhop|^2\right)
    = \Delta_{\Gamma_0}\rhop \, \nabla_{\Gamma_0}\rhop,
\end{equation*}
to put all the derivatives on $\Ve$.
Namely:
\begin{align*}
\langle \div(\rhop \Ve), \Delta_{\Gamma_0}\rhop \rangle
    &= \int_{\Gamma_0}{\Ve\cdot\nabla_{\Gamma_0}\rhop\, \Delta_{\Gamma_0}\rhop}
     + \int_{\Gamma_0}{\div_{\Gamma_0}(\Ve)\rhop\Delta_{\Gamma_0}\rhop} \\ 
    &= \int_{\Gamma_0}{\Ve\cdot \nabla_{\Gamma_0}\left(\tfrac{|\nabla_{\Gamma_0}\rhop|^2}{2}\right)}
     - \int_{\Gamma_0}{\nabla_{\Gamma_0}\left(\div_{\Gamma_0}(\Ve)\rhop\right)\cdot \nabla_{\Gamma_0}\rhop} \\
    &= - \frac{1}{2} \int_{\Gamma_0}{\div_{\Gamma_0}(\Ve) |\nabla_{\Gamma_0}\rhop|^2}
     - \int_{\Gamma_0}{\nabla_{\Gamma_0}\left(\div_{\Gamma_0}(\Ve)\rhop\right)\cdot \nabla_{\Gamma_0}\rhop}.
\end{align*}
At the end, we get
\begin{equation} \label{e:H1-est-anti-sym-term}
\left| \langle \div(\rhop \Ve), \Delta_{\Gamma_0}\rhop \rangle \right|
    \leq 2 \Vert\Ve\Vert_{W^{2,\infty}} \Vert\rhop\Vert_{H^1}^2.
\end{equation}
\item For the third term of \eqref{e:H1-est-test-funct}, we integrate by part on $\Gamma_0$, we have:
\begin{equation}
\langle \tilde{b_0}(\rhop,\psi), \Delta_{\Gamma_0}\rhop \rangle
     = \langle \nabla_{\Gamma_0}\tilde{b_0}(\rhop,\psi), \nabla_{\Gamma_0}\rhop \rangle.
\end{equation}
Since $\tilde{b_0}$ is of order $0$ with respect to $\rhop$, 
\begin{equation*}
\Vert \tilde{b_0}(\rhop,\psi)\Vert_{H^1} \leq \Vert \rhop\Vert_{H^1}.
\end{equation*}
Therefore
\begin{equation} \label{e:H1-est-bounde-term}
\left|\langle \tilde{b_0}(\rhop,\psi), \Delta_{\Gamma_0}\rhop \rangle\right|
     \leq C \Vert \rhop\Vert_{H^1}^2.
\end{equation}
\end{itemize}
By combining \eqref{e:H1-est-coerc-term}, \eqref{e:H1-est-anti-sym-term} and \eqref{e:H1-est-bounde-term} into \eqref{e:H1-est-test-funct}, and with the help of Lemma \ref{l:coerc-nab-dagger-nab}, we get the wanted inequality.
\end{proof}

\subsection{Leray scheme and well-posedness} \label{ss:wp}

\begin{thm}
Let $\psi\in  L^{\infty}([0,+\infty),W^{3,\infty}(\Gamma_0))$ be such that the function $\hw := (\Id+\A_0)\psi$ is bounded by below by a  positive constant $\hw\geq \alpha>0$.
Let $\rho_1\in H^1(\Gamma_0)$.
Then there exists an unique solution $\rhop$ to \eqref{e:prot-linearised} with initial value $\rhop(0,\cdot) = \rho_1$.
Moreover, this solution verifies, for all $t>0$, the following estimate
\begin{equation} \label{e:energ}
\Vert \rhop(t,\cdot)\Vert_{H^1}^2 
   \leq \exp(Ct)\Vert \rho_1\Vert_{H^1}^2.
\end{equation}
\end{thm}

\begin{rem}
Having a sign condition, similar to $\hw\geq \alpha>0$, is quite common in fluid dynamics.
For example the Rayleigh-Taylor criterion state that the gravity has to be downward for the water-waves equation to be well-posed (see \cite{Lannes}).
The same criterion also applies in the case of the Muskat equation (see \cite{Cordoba-Muskat-WPIP}), which is very close to ours.
Also, the authors in \cite{Cellule1} had a similar condition for well-posedness.
\end{rem}

\begin{proof}
We introduce $(v_n)_n$ an Hilbertian basis of the space $L^2(\Gamma_0)$, with the additional hypothesis that the $v_n$'s are the eigenvectors of the laplacian $\Delta_{\Gamma_0}$ (in particular, each $v_n$ is $C^{\infty}$).
We denote by $P_n$ the projection on $E_n:= \mathrm{Vect}\{v_1,\dots,v_n\}$
\begin{equation}
P_n(f) := \sum_{k=1}^{n}{\langle f,v_k\rangle v_k}.
\end{equation}
Once this is done, we study the projection of \eqref{e:prot-linearised} on $E_n$
\begin{equation} \label{e:prot-lin-approx-n}
\partial_t r_n + P_n \A_1(P_n r_n,\psi) = 0.
\end{equation}

\begin{lem}\label{l:approx-sol-exists}
There exists a unique solution $r_n$ of \eqref{e:prot-lin-approx-n} with initial value $P_n(\rho_1)$ and this solution verifies the inequality 
\begin{equation} \label{e:approx-energ}
\forall t>0, \quad
\Vert r_n(t,\cdot)\Vert_{H^1}^2 
   + c\alpha \int_{0}^{t}{\Vert r_n(s,\cdot)\Vert_{\dot{H}^{3/2}}^2 \exp(C(t-s)) \mathrm{d}s}
   \leq \exp(Ct)\Vert \rho_1\Vert_{H^1}^2.
\end{equation}
\end{lem}

\begin{proof}[Proof of Lemma \ref{l:approx-sol-exists}]
We write 
\begin{equation}
r_n(t,x) := \sum_{k=1}^{n}{a_{n,k}(t) v_k(x)},
\end{equation}
where the coefficient $a_{n,k}$ are still to be computed.

The equation \eqref{e:prot-lin-approx-n} can then be written in terms of the coefficients $(a_{n,k})_k$ as:
\begin{equation}
\forall k,\in [\![1,n ]\!], \quad
a_{n,k}'(t) = - \sum_{j=1}^{n}{\langle \A_1(v_j,\psi(t,\cdot)), v_k \rangle \, a_{n,j}(t)}.
\end{equation}
This is a linear ODE in $\R^n$, with bounded coefficients and therefore admits a unique global solution.

We have
\begin{equation}
\langle P_n \A_1(P_n r_n,\psi), r_n \rangle = \langle \A_1(r_n,\psi),r_n\rangle,
\end{equation}
and moreover, utilising the fact that the $P_n$ commutes with $\Delta_{\Gamma_0}$ (due to $(v_n)_n$ being a basis of eigenvalues of $\Delta_{\Gamma_0}$), we have also:
\begin{equation}
\langle P_n \A_1(P_n r_n,\psi), \Delta_{\Gamma_0} r_n \rangle 
    = \langle \A_1(r_n,\psi),\Delta_{\Gamma_0}r_n\rangle.
\end{equation}
Now since the function $r_n$ is $C^{\infty}_x$, we can apply Propositions \ref{p:energy-estimates} and \ref{p:H1-estimates} to get:
\begin{equation}
\tfrac{1}{2}\tfrac{d}{dt}\left(\Vert\rhop\Vert_{H^1}^2\right)
    + c\alpha \Vert\rhop\Vert_{H^{3/2}}^2
    \leq C\Vert\rhop\Vert_{H^1}^2.
\end{equation}
Then, by Gronwall's inequality, we get \eqref{e:approx-energ}.
\end{proof}

Let us fix $T>0$.
Now, since $\Vert r_n\Vert_{L^{2}([0,T], \dot{H}^{3/2}(\Gamma_0)}$ is bounded, by Banach-Alaoglu theorem, the sequence $(r_n)_n$ does converge weekly to a function $\rhop$ in $L^2([0,T], H^{3/2}(\Gamma_0))$ up to extraction.
And since $\A_1$ is linear, with its adjoint $\A_1$ well defined in $H^1$, with value in $L^2$, the function $\rhop$ is a solution of \eqref{e:prot-linearised} and it verifies \eqref{e:energ}.

Now, let us assume that there exists a second solution $r$ to \eqref{e:prot-linearised}.
Then since both $\rhop$ and $r$ lie in $L^{2}_tH^1_x$, we can apply Proposition \ref{p:energy-estimates} to $\rhop - r$ and get
\begin{equation}
\tfrac{1}{2}\tfrac{d}{dt}\left(\Vert r-\rhop\Vert_{L2}^2\right)
    + c\alpha \Vert r-\rhop\Vert_{H^{1/2}}^2
    \leq C\Vert r-\rhop\Vert_{L^2}^2.
\end{equation}
But both solutions have the same initial datum $\rho_1$
Therefore $r=\rhop$ and the solution $\rhop$ that we constructed is the only one.
\end{proof}

\begin{rem}
It would seem natural to try to prove the well-posedness in $L^2$ instead of $H^1$. 
The fact that we have natural $L^2$ estimate is a good hint that this property is also true.
However, to use the estimate of Proposition \ref{p:energy-estimates}, one need $\rhop$ to be $H^1$, which is not obvious at first glance.

To conclude with the kind of $L^2$ estimate that we have, for a linear equation, one can (and most of the time, it is what people do) prove that "every weak solution is a limit of regular solutions". 
Then due to the regular solution verifying the estimates, at the limit, the weak solution also does, which as we have seen above, prove uniqueness.

Now, we need to find a sequence of solution $r_{\varepsilon}$ that converges toward our weak solution $\rhop$ and here it gets tricky.
For most equations, a good candidate for such a sequence is to take a Fourier mollified sequence $r_{\varepsilon} := J_{\varepsilon}\rhop$.
But in this case, one would need to prove that the mollifier $J_{\varepsilon}$ converges nicely with $\G^i_0$ at low regularity, which we did not succeed in doing.

 Since the equation involves the Dirichlet-to-Neumann operator, one could also take $\Pi_n$ the projection onto the first $n$ eigenvectors of this operator, but then we are left with a commutator between this projection and $\div(\cdot V^e)$.

A maybe more suitable approach would be to say that we have not only a good estimate for Gronwall but also dissipation through $\Vert \rhop\Vert_{H^{1/2}}^2$.
So one could try to use the theory for accretive operator and regularise $\A_1$ using $(\lambda\Id +\A_1)^{-1}$.
We refer to \cite{Pazy}, Chapter 5, Paragraph 6 for the reader eager to use such an argument.
However, in our case, the regularisation does depend on $t$ and one need to control $\partial_t\A_1(\lambda\Id +\A_1)^{-1}$.
But the regularisation grants us only $1/2$ derivative, whereas the operator itself cost $1$ derivative, so one cannot conclude this method as well.
\end{rem}


\section{Perspectives}


The next obvious step is to treat the full equation \eqref{e:prot-equa} and not just the linearised equation around $0$.
We believe this can be done through a quasilinearisation method.
For example, one could write the equation for $\rhop = \partial_t\rho$, obtained by derivating \eqref{e:prot-equa} with respect to time
\begin{equation} \label{e:scheme-fin}
\partial_t \rhop 
    - d_{\rho}\A(\rhop)(|\Nor|\psi) 
    - \A[\rho]( \partial_t |\Nor|  \, \psi)
    = \A[\rho](|\Nor| \partial_t\psi).
\end{equation}
This equation, for $\rho$ fixed, is more or less the linearised equation \eqref{e:prot-linearised}.
With good enough estimates, one could attempt to construct a solution through the scheme :
\begin{itemize}
\item $\rho^0=0$,
\item $\rhop^n$ is the solution of \eqref{e:scheme-fin} for $\rho = \rho^{n-1}$,
\item $\rho^n = \rho_0 + \int_{0}^{t}{\rhop^n(s,\cdot) ds}$.
\end{itemize}
The difficulty with such a scheme is to ensure that, at each point in the algorithm we have good estimates on $\rhop^n$ and $\rho^n$.
However, our analysis of the equation is not powerful enough here.
We treated the linearized equation around zero, and one could treat similarly the equation linearised around any smooth function.
However, since $\rhop^n$ is only guaranteed to be $H^1$, $\rho^n$ is of this regularity as well.
So there is still much to do in order to treat the well-posedness of~\eqref{e:prot-equa}.

\ \par \ 

\paragraph{\textbf{Acknowledgements.}}
The author is partially supported by 
the Agence Nationale de la Recherche,
Project SINGFLOWS, 
ANR-18-CE40-0027-01.
The author warmly thank David Lannes for the careful reading and advises,
and Clair Poignard for the interesting discussions on the subject.

\end{document}